\documentclass[reqno]{amsart}

\usepackage[pagewise]{lineno}

\usepackage{color}

\usepackage{hyperref}

\newtheorem{theorem}{Theorem}
\newtheorem{lemma}[theorem]{Lemma}

\newtheorem{definition}[theorem]{Definition}
\newtheorem{proposition}[theorem]{Proposition}
\newtheorem{remark}[theorem]{Remark}

\newtheorem{example}[theorem]{Example}

\numberwithin{theorem}{section}
\numberwithin{equation}{section}

\begin{document}

\title[Continuity and topological structural stability]{Continuity and topological structural stability for nonautonomous random attractors}

\author[T. Caraballo]{Tom\'as Caraballo$^1$}
\email{caraball@us.es}
\thanks{$^1$ Departamento de Ecuaciones Diferenciales y An\'alisis Num\'erico, Universidad de Sevilla, Campus Reina Mercedes, 41012, Sevilla, Spain.}
\author[A. N. Carvalho]{Alexandre N. Carvalho$^2$}
\thanks{$^2$ Instituto de Ci\^encias Matem\'aticas e de Computa\c c\~ao, Universidade de S\~ ao Paulo, Av. Trab. São Carlense, 400 - Centro, São Carlos - SP, 13566-590 Brazil.}
\email{andcarva@icmc.usp.br}

\author[J. A. Langa]{Jos\'e A. Langa$^1$}
\email{langa@us.es}

\author[A. N. Oliveira-Sousa]{Alexandre N. Oliveira-Sousa$^{2}$}
\email{alexandrenosousa@gmail.com}



\date{}
\begin{abstract}
	In this work, we study continuity and topological structural stability of attractors for nonautonomous random differential equations obtained by small bounded random perturbations 
	of autonomous semilinear problems. First, we study existence and permanence of unstable sets of hyperbolic solutions. 
	Then, we use this to establish lower semicontinuity of nonautonomous random attractors
	and to show that the gradient structure persists under nonautonomous random perturbations. 
	Finally, we apply the abstract results in a stochastic differential equation and in a damped wave equation with a perturbation on the damping.
\end{abstract}

\maketitle

\section{Introduction}

In this paper, we study autonomous attractors under nonautonomous random perturbations. Our goal is to provide conditions to conclude continuity and topological structural stability of nonautonomous random attractors. We consider an autonomous semilinear problem in a Banach space $X$
\begin{equation}\label{newint-autonomous-differential-equation}
\dot{y}=By+f_0(y), \ \  t>0, \ y(0)=y \in X,
\end{equation} 
and its nonautonomous random perturbations of the type
\begin{equation}\label{newint-eq-nonautonomous-random-perturbation}
\dot{y}=By+f_\eta(t,\theta_t\omega,y), \  t>\tau, \ y(\tau)=y_\tau \in X, \ \eta\in (0,1],
\end{equation}  
where $B$ generates a $C^0$-semigroup $\{e^{At}: t\geq 0\}\subset \mathcal{L}(X)$, and $\theta_t:\Omega\to \Omega$ is a random flow defined in a probability space $(\Omega,\mathcal{F},\mathbb{P})$. 
\par We assume that problem \eqref{newint-autonomous-differential-equation} generates a (nonlinear) semigroup $\{T(t): t\geq 0 \}$, and that \eqref{newint-eq-nonautonomous-random-perturbation} generates a (nonlinear) nonautonomous random dynamical system $(\psi_\eta,\Theta)$, for each $\eta\in [0,1]$, and that all these dynamical systems have attractors, 
see \cite{Bixiang-Wang-existence,Bixiang-Wang-existence-stability,Bixiang-Wang-existence-upper} and the references therein for general theory and examples. 

\par One of our goals is to establish continuity of this family of  attractors. 
This is done by proving \textit{upper} and \textit{lower semicontinuity}. 
On one hand, \textbf{upper semicontinuity} means that the perturbed attractors do not become suddenly much larger than the limiting attractor (non-explosion). 
On the other hand, \textbf{lower semicontinuity} means that the perturbed attractors do not become suddenly much smaller than the limiting attractor (non-implosion). 
For an introduction to the notion of continuity of attractors see
\cite[Chapter 3]{Carvalho-Langa-Robison-book} for global and pullback attractors, and \cite[Section 4.10]{Hale-book-88} or
\cite[Chapter 8]{Babin-Vishik-book} for global attractors.
\par For nonautonomous (deterministic) dynamical systems the continuity of attractors is very well studied, 
see for instance \cite{Bortolan-Carvalho-Langa,Carvalho-Langa-2,Carvalho-Langa-Robinson,Langa-Robinson-Suarez-Vidal-Lopez-07}. In the nonautonomous random setting, the upper semicontinuity was proved in several examples, see \cite{BATES201432,Bixiang-Wang-existence,Bixiang-Wang-existence-upper} and the references therein.
However, the lower semicontinuity is more difficult to attain due to the fact that one has to prove that the \textit{inner structure} of the limiting attractor is ``preserved'' under perturbation, in order to ensure that the perturbed attractor occupies a region `as large as' the region occupied by the limiting attractor.
More precisely, the typical conditions one has to assume is that the limiting attractor is the union of the unstable sets of the equilibria and then give conditions to ensure that these equilibria and their unstable sets `persist' under perturbation, see \cite{Arrieta-Carvalho-04,Babin-Vishik-83,Bruschi-Carvalho-Cholewa,Hale-Raugel-89} for the lower semicontinuity of global attractors, and \cite{Carvalho-Langa-Robinson,Carvalho-Langa-2,Langa-Robinson-Suarez-Vidal-Lopez-07} for the lower semicontinuity of pullback attractors and \cite{Bortolan-Carvalho-Langa} for the lower semicontinuity of uniform attractors. 
In \cite{Carvalho-Langa-2} the authors study the permanence of hyperbolic global solutions and of their corresponding unstable and stable sets, in the nonautonomous setting, and in \cite{Carvalho-Langa-Robinson} the authors prove a general result on the lower semicontinuity of pullback attractors allowing the limiting pullback attractor to be given as the closure of a countable (possibly infinity) union of unstable sets of hyperbolic global solutions.

\par Thus, to prove lower semicontinuity in a nonautonomous random framework we follow this method of proving that the inner structure persists under perturbation. 
However, this is not expected to happen for general types of noises.
Actually, some works provides that the presence of an additive noise destroys the continuity of the attractors \cite{Bianchi-Blomker-Yang-16,Crauel-Flandoli-98}, see also \cite{Callaway-Doan-Lamb-Rasmussen} for an complementary study of such problems.
Hence, to obtain our results we will consider small bounded random perturbations as the one introduced in \cite{Caraballo-Carvalho-Langa-OliveiraSousa-2021}, where the authors studied the existence and permanence of hyperbolic solutions for 
\eqref{newint-eq-nonautonomous-random-perturbation} assuming that the perturbations are uniformly bounded in time. Now, inspired by the results in \cite{Carvalho-Langa-2}, 
we study the existence and continuity of the unstable sets associated with this hyperbolic solutions, and we use these results to conclude the lower semicontinuity for the attractors of 
$\{(\psi_\eta,\Theta): \eta\in [0,1]\}$, see Theorem \ref{th-continuity-nonautonomous-random-attractors}. 
In our proofs, we show how to control the random parameter using techniques of deterministic dynamical systems. 

\par The idea of reproducing the internal structure in the perturbed attractor is not only important to show continuity of attractors, 
but is also crucial to prove that the dynamics are preserved under perturbation.
For instance, in \cite{Carvalho-Langa-1} the authors provide conditions (permanence of the inner structure) to prove that dynamically gradient semigroups are stable under perturbation.
We refer to this property as \textbf{topological structural stability}.
Gradient dynamical systems were widely studied in the past years, see  \cite{Aragao-Costa-2013,Bortolan-Cardoso-etal,Bortolan-Carvalho-Langa,Bortolan-Carvalho-Langa-book,Carvalho-Langa-2,Caraballo-Carv-Langa-Rivero-gradient} for deterministic dynamical systems, and \cite{Caraballo-Langa-Liu-12,Ju-QI-JWang-18} for random dynamical systems.
In this work, we obtain a result on the topological structural stability for nonautonomous random differential equations, see Theorem \ref{th-topological-structure-stability}. This will be also a consequence of the careful study of the internal structure of these attractors.
\par We also obtain stronger results on the continuity and topological structural stability of nonautonomous random attractors for the case when the random perturbations are uniformly bounded with respect to the random parameter, see Remark \ref{remark-uniformly-bounded-noises-continuity} and Remark \ref{remark-uniformly-bounded-noises-top-struct-stab} for more details. Moreover, see \cite{Bobryk-21,Caraballo-Cruz-21} for examples of this types of noises.
%

\par We provide two applications of our abstract results. 
First, in a family of stochastic differential equations with a nonautonomous multiplicative white noise
\begin{equation}\label{eq-newintroduction-stochastic}
dy=Bydt+f_0(y)dt+\eta\kappa_t y\circ dW_t, \ t\geq \tau, \ y(\tau)=y_\tau\in X,
\end{equation}
where $\eta\in [0,1]$, and the mapping $\mathbb{R}\ni t\mapsto \kappa_t\in \mathbb{R}$ is a real function that ``controls'' the growth of the noise in time, see Subsection \ref{subsec-stochastic-diff-equations}. Finally, a nonautonomous random perturbation on the damping of a damped wave equation with Dirichlet boundary condition
\begin{equation}\label{eq-newint-damped-wave}
u_{tt}+\beta_\eta(t,\theta_t\omega) u_t-\Delta u=f(u), \ t\geq \tau, \ \eta\in [0,1],
\end{equation}
where $\{\theta_t:\Omega\to \Omega: t\in \mathbb{R}\}$ is a random flow in a probability space 
$(\Omega,\mathcal{F},\mathbb{P})$ and there exists $\beta>0$ such that $\beta_\eta$ converges to $\beta$ as 
$\eta\to 0$, see Subsection \ref{subsec-app-dff-equations-wave}.
\par Next, we describe how the paper is organized. In Section \ref{sec:preliminaries}, we recall some basic concepts of the theory of nonautonomous and random dynamical systems. Then, in Section \ref{section-randomperturbation-of-autonomous-equilibria}, we present the results on the permanence of hyperbolic solutions and in Section \ref{sec:existence-continuity-unstable-sets}, we obtain the existence and continuity of local unstable sets associated with these solutions. In Section \ref{sec:continuity-of-NR-Attractors},
we prove our result on the continuity of nonautonomous random attractors. In Section \ref{sec:top-struc-stab}, we provide a result on the topological structural stability. Finally, in Section \ref{sec-app-dff-equations}, we present applications to differential equations.

\section{Preliminaries}
\label{sec:preliminaries}

\par First, we introduce the notion of  \textit{nonautonomous random dynamical systems} 
in 
a complete separable metric space $(X,d)$. 
\begin{definition}
	Let $(\Omega,\mathcal{F},\mathbb{P})$ be a probability space. 
	We say that a family of maps $\{\theta_t:\Omega\rightarrow\Omega: \, t\in\mathbb{R}\}$ is a
	\textbf{random flow} if 
	\begin{enumerate}
		\item $\theta_0=Id_\Omega$;
		\item $\theta_{t+s}=\theta_t\circ \theta_s$, for all $t,s\in\mathbb{R}$;
		\item $\theta_t:\Omega \rightarrow \Omega$ is measurable for all $t\in\mathbb{R}$.
	\end{enumerate}
\end{definition}
\begin{definition}
	Let $\{\theta_t:\Omega\rightarrow\Omega: \, t\in \mathbb{R}\}$ be a random flow. 
	Define $\Theta_t(\tau,\omega):=(t+\tau,\theta_t\omega)$ for each $(\tau,\omega)\in\mathbb{R}\times\Omega$, and $t\in\mathbb{R}$.
	We say that a family of maps $\{\psi(t,\tau,\omega):X\to X; (t,\tau,\omega)\in\mathbb{R}^+\times\mathbb{R}\times\Omega\}$ 
	is a \textbf{nonautonomous random dynamical system (co-cycle)} driven by
	$\Theta$ if
	\begin{enumerate}
		\item the mapping
		$\mathbb{R}^+\times\Omega\times X\ni (t,\omega,x)\mapsto \psi(t,\tau,\omega)x\in X$
		is measurable
		for each fixed $\tau\in\mathbb{R}$;
		\item
		$\psi(0,\tau,\omega)=Id_X$,
		for each $(\tau,\omega)\in\mathbb{R}\times\Omega$;
		\item 
		$\psi(t+s,\tau,\omega)=\psi(t,\Theta_s(\tau,\omega))\circ\psi(s,\tau,\omega)$,
		for every $t,s\geq 0$ in 
		$\mathbb{R}$, and $(\tau,\omega)\in\mathbb{R}\times \Omega$;
		\item $\psi(t,\tau,\omega):X\to X$ is a continuous map for each  $(t,\tau,\omega)\in\mathbb{R}^+\times\mathbb{R}\times\Omega$.
	\end{enumerate}
	We usually denote the pair $(\psi,\Theta)_{(X,\mathbb{R}\times\Omega)}$, or $(\psi,\Theta)$, to denote the co-cycle $\psi$ driven by $\Theta$.
\end{definition}

\begin{remark}
	We will write $\omega_\tau:=(\tau,\omega)\in \mathbb{R}\times \Omega$, and
	$\Theta_t(\omega_\tau):=(t+\tau,\theta_t \omega)=(\theta_t\omega)_{\tau+t}$. 
\end{remark}

\par Throughout this work we will assume that a nonautonomous random dynamical system $(\psi,\Theta)$ satisfies
\begin{equation}\label{eq-aditional-assumption-on-the-cocycles}
\mathbb{R}^+\times X\ni (t,x)\mapsto \psi(t,\omega_{\tau})x\in X \hbox{ is continuous, for each } \omega_\tau\in \mathbb{R}\times \Omega.
\end{equation}
This assumption is sensible in the applications, e.g., when the co-cycle is induced by a well-posed stochastic/random differential equation. Hence, we can associate our co-cycle with a family of \textit{evolution processes}. Recall that:

\begin{definition}
	Let $\mathcal{S}=\{S(t,s) ; \, t\geq s\}$ be a family of continuous operators from $X$ into itself.
	We say that $\mathcal{S}$ is an \textbf{evolution process} in $X$ if
	 $S(t,t)=Id_X$, for all $t\in \mathbb{R}$,
	$S(t,s)S(s,\tau)=S(t,\tau)$, for $t\geq s\geq \tau$, and the mapping
	$\{(t,s)\in \mathbb{R}^2 ; \, t\geq s\}\times X\ni (t,s,x)\mapsto S(t,s)x$ is continuous.
\end{definition}

\begin{remark}\label{remark-associated-processes}
	Let $(\psi,\Theta)_{(X,\mathbb{R}\times \Omega)}$ be a nonautonomous random dynamical system which satisfies \eqref{eq-aditional-assumption-on-the-cocycles}.
	Then, for each $\omega_\tau\in \mathbb{R}\times\Omega$, we define the following \textit{evolution process} 
	\begin{equation*}
		\Psi_{\omega_\tau}:=\{\psi(t-s,\Theta_s\omega_\tau)\,; \, t\geq s \}.
	\end{equation*}
\end{remark}

\begin{definition}\label{definition-measurable-set-valued-maps}
	Let $K:\Omega\to 2^X$ be a set-valued mapping with closed nonempty images. We say that $K$ is \textbf{measurable} if the mapping
	$\Omega\ni \omega \mapsto d(x,K(\omega))$ is $(\mathcal{F}, \mathcal{B}_\mathbb{R})$-measurable for every fixed $x\in X$.
\end{definition}

\par In Definition \ref{definition-measurable-set-valued-maps}, we used that $X$ is a complete separable metric space, see \cite[Chapter III]{CastaingValadier}.


\begin{definition}
	Let $\hat{\mathcal{A}}=\{\mathcal{A}(\omega_\tau): \omega_{\tau}\in \mathbb{R}\times \Omega\}$ be a family of nonempty subsets of $X$. 
	We say that $\hat{\mathcal{A}}$ is a \textbf{nonautonomous random attractor} for $(\psi,\Theta)$
	if the following conditions are fulfilled:
	\begin{enumerate}
		
		\item $\mathcal{A}(\omega_{\tau})$ is compact, for every $\omega_\tau\in \mathbb{R}\times \Omega$;
		\item the set-valued mapping $\omega\mapsto \mathcal{A}(\tau,\omega)$ is measurable, for each $\tau \in \mathbb{R}$;
		\item $\hat{\mathcal{A}}$ is invariant, i.e., 
		$\psi(t,\omega_{\tau})\mathcal{A}(\omega_{\tau})=
		\mathcal{A}(\Theta_t\omega_{\tau})$ for every $t\geq 0$ and $\omega_{\tau}\in \mathbb{R}\times \Omega$;
		\item $\hat{\mathcal{A}}$ pullback attracts every bounded subset of $X$, i.e.,
		for every bounded subset $B$ of $X$ and $\omega_{\tau}\in \mathbb{R}\times\Omega$,
		\begin{equation*}
			\lim_{t\to +\infty}dist(\psi(t,\Theta_{-t}\omega_\tau)B, \mathcal{A}(\omega_{\tau}))=0,
		\end{equation*}
	where $dist(A,B)=\sup_{a\in A}\inf_{b\in B} d(a,b)$ is the usual Hausdorff semi-distance;
	\item $\hat{\mathcal{A}}$ is the minimal closed family that 
	pullback attracts bounded subsets of $X$, i.e., if 
	$\{F(\omega_{\tau}): \omega_\tau\in \mathbb{R}\times \Omega\}$ is a family of closed subsets of $X$ that pullback attracts every bounded subset of $X$, then $\mathcal{A}(\omega_{\tau})\subset F(\omega_{\tau})$, for every $\omega_{\tau}\in \mathbb{R}\times \Omega$.
	\end{enumerate}
	
\end{definition}

\par For existence of nonautonomous random attractors and applications to differential equations, see Wang \cite{Bixiang-Wang-existence} and the references therein.
%

\par Since we will associated our co-cycle $(\psi,\Theta)$ with a family of evolution processes as in Remark \ref{remark-associated-processes}, we recall the notion of \textit{pullback attractors}.

\begin{definition}
	Let $\mathcal{S}=\{S(t,s): t\geq s\}$ be an evolution process in $X$ and 
	$\{\mathcal{A}(t): t\in \mathbb{R}\}$ be a family of nonempty subsets of $X$. 
	We say that $\{\mathcal{A}(t): t\in \mathbb{R}\}$ is a \textbf{pullback attractor} for $\mathcal{S}$
	if
	\begin{enumerate}
		\item $\mathcal{A}(t)$ is compact, for every $t\in \mathbb{R}$;
		\item $\{\mathcal{A}(t): t\in \mathbb{R}\}$ is invariant, i.e., 
		$S(t,s)\mathcal{A}(s)=\mathcal{A}(t), \ \forall \, t\geq s$;
		\item $\{\mathcal{A}(t): t\in \mathbb{R}\}$ pullback attracts every bounded subset of $X$, i.e.,
		for every bounded subset $B$ of $X$,
		\begin{equation*}
			\lim_{s\to -\infty}dist(S(t,s)B, \mathcal{A}(t))=0;
		\end{equation*}
		\item $\{\mathcal{A}(t): t\in \mathbb{R}\}$ is the minimal closed family that 
		pullback attracts bounded subsets of $X$.
	\end{enumerate}
\end{definition}
\par There are several works that deal with existence and continuity (upper and lower semicontinuity) of pullback attractors, we refer the reader to \cite{Caraballo-Lukaszewicz-Real-06,Caraballo-Carv-Langa-Rivero-10,Carvalho-Langa-Robison-book,Bortolan-Carvalho-Langa-book}, where many other references to earlier results can be found.

\begin{remark}
	Let $(\psi,\Theta)$ be a nonautonomous random dynamical system with an attractor
	$\{\mathcal{A}(\omega_{\tau}): \omega_{\tau}\in \mathbb{R}\times \Omega\}$. 
	Then, for each $\omega_{\tau}$ fixed, the evolution process $\Psi_{\omega_{\tau}}$ has 
	a pullback attractor given by 
	$\{A(\Theta_t\omega_{\tau}): t \in \mathbb{R}\}$.
\end{remark}

Finally, we recall the definition of the \textit{unstable set} for a global solution $\xi$ of an evolution process, which was introduced in \cite{Carvalho-Langa-2}.
\begin{definition}
	Let $\mathcal{S}=\{S(t,s): t\geq s\}$ be an evolution process, and 
	$\xi:\mathbb{R}\to X$ be a \textbf{global solution} of 
	$\mathcal{S}$, i.e., $S(t,s)\xi(s)=\xi(t)$, for every $t\geq s$.
	The \textbf{unstable set} of $\xi$ is defined as
	\begin{eqnarray*}
		W^u(\xi)=\bigg\{(t,z)\in \mathbb{R}\times X: 
		\hbox{ there is a global solution } \zeta \hbox{ of } 
		\mathcal{S} \hbox{ such that }\\
		\zeta(t)=z, \hbox{ and }
		\lim_{s\to -\infty} \|\zeta(s)-\xi(s)\|_X=0
		\bigg\}.
	\end{eqnarray*}
{\bf The section of $W^u(\xi)$ at time $t\in \mathbb{R}$} is denoted by
$W^u(\xi)(t)=\{z\in X: (t,z)\in W^u(\xi)\}$.
\end{definition}

\begin{remark}
	Let $\mathcal{S}=\{S(t,s): t\geq s\}$ be an evolution process with a pullback attractor
	$\{\mathcal{A}(t): t\in \mathbb{R}\}$ such that $\cup_{t\leq 0}\mathcal{A}(t)$ is bounded.
	In this case, 
	\begin{equation}
	\mathcal{A}(t)=\bigcup \{W^u(\xi)(t): \xi \hbox{ is a backwards-bounded solution} \}, \ \forall t\in \mathbb{R},
	\end{equation}
	where $\xi$ is \textit{backwards-bounded} means that the set $\xi(-\infty,0]$ is bounded.
	Therefore, it is natural to search for the minimal collection of backwards-bounded solutions 
	whose unstable sets form the attractor. 
	Of course many backwards-bounded solutions have the same unstable set, and thus 
	it is natural to seek for backward-separated solutions, see \cite[Section 3.3]{Carvalho-Langa-Robison-book} for more details.
	In Section \ref{sec:top-struc-stab}, we will provide conditions to obtain that there is a 
	distinguished 
	set of backwards-bounded global solutions that forms the nonautonomous random attractor. 
	These conditions rely on the hyperbolicity, which we will study in the following sections. It is through this distinguished set that we will be able to address the lower semicontinuity of nonautonomous random attractors. 
\end{remark}

\section{Permanence of random hyperbolic solutions}
\label{section-randomperturbation-of-autonomous-equilibria}

\par In this section, we recall some results on the existence and continuity of hyperbolic solutions for nonautonomous random differential equations obtained in \cite{Caraballo-Carvalho-Langa-OliveiraSousa-2021}. As we will see further, these results are crucial to obtain the lower semicontinuity and topological structural stability of attractors.

\par As in \cite[Section 3]{Caraballo-Carvalho-Langa-OliveiraSousa-2021}, problems
\eqref{newint-autonomous-differential-equation} and \eqref{newint-eq-nonautonomous-random-perturbation} can be seen as the following family of semilinear differential equations on a separable Banach space $X$
\begin{eqnarray}
& &\dot{y}=By+f_0(y), \ \ y(0)=y_0,
\label{eq-autonomous-semilinear-ODE} \\
& &\dot{y}=By+f_\eta(\Theta_t\omega_\tau,y), \ \ y(0)=y_0,
\label{eq-nonautonomous-semilinear-randomODE}
\end{eqnarray}
where $\{\Theta_t: t\in \mathbb{R}\}$ is a driving flow given by 
$\Theta_t(\omega_\tau):=(t+\tau,\theta_t\omega)$ for every $\omega_\tau=(\tau,\omega)\in \mathbb{R}\times \Omega$.
\par We suppose that $f_\eta(\omega_{\tau},\cdot)\in C^1(X)$, for every $\eta\in [0,1]$, $\omega_{\tau}\in \mathbb{R}\times \Omega$, and that
\begin{equation}\label{eq-condition-f_eta-f_0}
\lim_{\eta\to 0}\sup_{(t,x)\in\mathbb{R}\times B(0,r)}\Big\{\|f_\eta(\Theta_t\omega_\tau,x)-f_0(x)\|_X+
\|(f_\eta)_x(\Theta_t\omega_\tau,x)-f_0^\prime(x)\|_{\mathcal{L}(X)}\Big\}= 0,
\end{equation}
for all $r\geq 0$ and $\omega_\tau\in \mathbb{R}\times\Omega$. 
This ensures local well-posedness and differentiability with respect to the initial conditions of \eqref{eq-autonomous-semilinear-ODE} and \eqref{eq-nonautonomous-semilinear-randomODE}, for each $\omega_{\tau}\in \mathbb{R}\times \Omega$.
\par Additionally, we assume global well-posedness, so that \eqref{eq-autonomous-semilinear-ODE} is associated with a semigroup $\mathcal{T}_0=\{T_{0}(t): t\geq 0 \}$, and that \eqref{eq-nonautonomous-semilinear-randomODE} is associated with nonautonomous random dynamical system $(\psi_\eta,\Theta)$, for each $\eta\in [0,1]$. 
 \par We say that a map $\xi:\mathbb{R}\times\Omega\to X$ is a \textbf{global solution} for $(\psi,\Theta)$
	if
	\begin{equation*}
	\psi(t,\omega_\tau)\xi(\omega_\tau)=\xi(\Theta_t\omega_\tau), \hbox{ for every }
	t\geq 0.
	\end{equation*}
   Then, for each $\omega_\tau$ fixed, the mapping 
	$\mathbb{R}\ni t\mapsto \xi(\Theta_t\omega_\tau)$
	defines a global solution for the evolution process 
	$\{\psi(t-s,\Theta_s\omega_\tau):t\geq s \}$.

	\par We are interested in the global solutions that are \textit{hyperbolic}, and to 
	define hyperbolic solutions we need to recall the concept of exponential dichotomy. 
	First, recall the definition of $\Theta$-\textit{invariance}:
	\begin{definition}
		A map $M:\mathbb{R}\times \Omega\to \mathbb{R}$ is said to be \textbf{$\Theta$-invariant} if
		for each $\omega_\tau\in \mathbb{R}\times\Omega$ we have that
		$M(\Theta_t\omega_\tau)=M(\omega_{\tau})$, for every $t\in\mathbb{R}$.
	\end{definition}
\begin{definition}
	A linear nonautonomous random dynamical system $(\varphi,\Theta)$ such that
	$\varphi(t,\tau,\omega)\in \mathcal{L}(X)$, for all 
	$(t,\tau,\omega)\in \mathbb{R}^+\times\mathbb{R}\times\Omega$, is said to admit
	an \textbf{exponential dichotomy} if 
	there exists a subset $\tilde{\Omega}$ of $\Omega$ which $\theta_t\tilde{\Omega}=\tilde{\Omega}$ and 
	$\mathbb{P}(\tilde{\Omega})=1$, and a family of projections,
	$\Pi^s:=\{\Pi^s(\omega_\tau): \omega_\tau\in \mathbb{R}\times \tilde{\Omega}\}$
	such that
	\begin{enumerate}
		\item the map $\Pi^s(\tau,\cdot):\tilde{\Omega}\to \mathcal{L}(X)$ is strongly measurable, for each $\tau\in \mathbb{R}$;
		\item $\Pi^s(\Theta_t\omega_\tau)\varphi(t,\omega_\tau)=
		\varphi(t,\omega_\tau)\Pi^s(\omega_\tau)$, for every
		$t\in \mathbb{R}^+$ and $\omega_\tau\in \mathbb{R}\times \tilde{\Omega}$;
		\item $\varphi(t,\omega_\tau):R(\Pi^u(\omega_\tau))\to R(\Pi^u(\Theta_t\omega_\tau)) $ is an isomorphism, where
		$\Pi^u(\omega_{\tau}):=Id_X-\Pi^s(\omega_{\tau})$, for all $\omega_{\tau}\in \mathbb{R}\times \Omega$;
		\item there exist $\Theta$-invariant maps
		$\alpha:\mathbb{R}\times\Omega\to (0,+\infty)$ and $M:\mathbb{R}\times\Omega\to [1,+\infty)$
		such that 
		\begin{eqnarray*}
			\|\varphi(t,\omega_\tau)\Pi^s(\omega_\tau)\|_{\mathcal{L}(X)}
			&\leq&
			M(\omega_\tau)e^{-\alpha(\omega_\tau)t}, \hbox{ for every } t\geq 0;\\
			\|\varphi(t,\omega_\tau)\Pi^u(\omega_\tau)\|_{\mathcal{L}(X)}
			&\leq&
			M(\omega_\tau)e^{\alpha(\omega_\tau)t}, \hbox{ for every } t\leq 0,
		\end{eqnarray*}
		for every $\omega_\tau\in \mathbb{R}\times\tilde{\Omega}$.
	\end{enumerate}
	In this case, the function $M$ is called a \textbf{bound} and
	$\alpha$ an \textbf{exponent} for the exponential dichotomy.
\end{definition}

For nonautonomous random dynamical systems this notion was introduced by \cite{Caraballo-Carvalho-Langa-OliveiraSousa-2021}. Also in 
\cite{Caraballo-Carvalho-Langa-OliveiraSousa-2021} the authors proved a robustness result and as an application they established existence and continuity of random hyperbolic solutions for \eqref{eq-nonautonomous-semilinear-randomODE}. 

\par Recall that $y_0^*\in X$ is a \textbf{hyperbolic equilibrium} for 
\eqref{eq-autonomous-semilinear-ODE} if the linear operator
$A:=B+f^\prime_0(y_0^*)$ generates a $C^0$-semigroup 
$\{e^{At}:t\geq 0\}$ that admits an exponential dichotomy. 
We say that $\xi$ is a \textbf{random hyperbolic solution} of \eqref{eq-nonautonomous-semilinear-randomODE} if 
	the linearized nonautonomous random dynamical system
	$(\varphi,\Theta)$, given by
	\begin{equation*}
	\varphi(t,\omega_\tau)=e^{Bt}+
	\int_{0}^te^{B(t-s)}D_xf_\eta(\Theta_s\omega_\tau,\xi(\Theta_s\omega_\tau))
	\varphi(s,\omega_\tau)ds, \forall \omega_\tau \in 
	\mathbb{R}\times \Omega,
	\end{equation*}
	admits an exponential dichotomy.

\par Now, we present a result on the permanence of hyperbolic solutions,
for the proof see \cite[Theorem 3.9]{Caraballo-Carvalho-Langa-OliveiraSousa-2021}. 

\begin{theorem}[Existence and continuity of hyperbolic solutions]
	\label{th-existence-hyperbolic-solutions-complete}
	Let $y_0^*$ be a hyperbolic equilibrium for \eqref{eq-autonomous-semilinear-ODE} and assume that \eqref{eq-condition-f_eta-f_0} holds. 
	Given $\epsilon>0$ suitable small, there exists a $\Theta$-invariant map
	$\eta_\epsilon:\mathbb{R}\times\Omega\to (0,1]$ such that:
	\begin{enumerate}
		\item 
		for each $\omega_{\tau}\in \mathbb{R}\times \Omega$ fixed, 
		given $\eta\in (0,\eta_\epsilon(\omega_\tau)]$,
		there exists a global hyperbolic solution 
		$\mathbb{R}\ni t\mapsto \zeta_\eta(t,\omega_{\tau})$
		of the evolution process
		$\{\psi_{\eta}(t-s,\Theta_s\omega_{\tau}): t\geq s\}$ satisfying
		\begin{equation}\label{eq-continuity-hyperbolic-solutions}
			\sup_{t\in \mathbb{R}}\|\zeta_{\eta}^*(t, \omega_\tau)-y_0^*\|_X<\epsilon,
		\end{equation}
		and $\zeta_\eta(t,\omega_{\tau})=\zeta_\eta(0,\Theta_t\omega_{\tau})$, for all $t\in \mathbb{R}$.
	\item  
	for each $\Theta$-invariant function
	$\bar{\eta}:\mathbb{R}\times \Omega\to [0,1]$ 
	with $\bar{\eta}(\omega_{\tau})\leq \eta_\epsilon(\omega_{\tau})$,
	there exists
	a random hyperbolic solution $\xi_{\bar{\eta}}^*:\mathbb{R}\times \Omega\to X$ of 
	$(\psi_{\bar{\eta}},\Theta)$ defined by
	\begin{equation*}
		\xi_{\bar{\eta}}^*(\omega_{\tau}):=\zeta_{\bar{\eta}(\omega_\tau)}^*(0,\omega_\tau),
	\end{equation*} 
	and satisfying \eqref{eq-continuity-hyperbolic-solutions}.
	\end{enumerate}
\end{theorem}

\par Theorem \ref{th-existence-hyperbolic-solutions-complete} is the first step to the study of existence and continuity of unstable and stable sets, which are the main tool to conclude lower semicontinuity and topological structural stability of attractors.
%

\begin{remark}
	Suppose that $\{y_1^*,\cdots,y_p^*\}$ is a set of hyperbolic equilibria for \eqref{eq-autonomous-semilinear-ODE}. 
	Then there exists $\epsilon_0>0$ such that $y_i^*$ is \textit{isolated} in $B(y_i^*,\epsilon_0)$ and
	$B(y_i^*,\epsilon_0)\cap B(y_j^*,\epsilon_0)=\emptyset$, $j\neq i$.
	Theorem \ref{th-existence-hyperbolic-solutions-complete} guarantees that for each $i\in \{1,\cdots,p\}$ and $\epsilon_0'\in (0,\epsilon_0)$ suitable small fixed, 
	there exits a $\Theta$-invariant function
	$\eta_{0,i}: \mathbb{R}\times\Omega\to (0,1]$ 
	satisfying the conclusions of Theorem 
	\ref{th-existence-hyperbolic-solutions-complete}.
	\par Define 
	$\eta_0(\omega_{\tau})=\min_{0\leq i\leq p}\{\eta_{0,i}(\omega_{\tau})\}$, for $\omega_{\tau}\in \mathbb{R}\times \Omega$. 
	Let $\omega_\tau$ be fixed, then for each $\eta\in (0,\eta_{0}(\omega_{\tau})]$
	there exists
	$\zeta_{i,\eta}^*(\cdot,\omega_\tau)$ is a hyperbolic solution of $\{\psi_\eta(t-s,\Theta_s\omega_{\tau}) :t\geq s \}$ such that 
	\begin{equation*}
		\sup_{t\in \mathbb{R}}\|\zeta_{i,\eta}^*(t,\omega_\tau)-y_i^*\|_X<\epsilon_0', \hbox{ for every } i\in \{1,\cdots,p\}.
	\end{equation*}
	
\end{remark}

\section{Existence and continuity of unstable sets}
\label{sec:existence-continuity-unstable-sets}
\par In this section, we study existence and continuity of unstable sets for the hyperbolic solutions obtained in Theorem \ref{th-existence-hyperbolic-solutions-complete}[Item (1)]. 
Under the same assumptions of Section \ref{section-randomperturbation-of-autonomous-equilibria}, we will apply the techniques of the deterministic case \cite{Carvalho-Langa-2} to our problem. The idea here is to revisit the proofs to track the dependence on the parameter $\omega_\tau\in \mathbb{R}\times \Omega$ in the arguments.
\par First, inspired in \cite{Carvalho-Langa-2}, we extend the concept of\textit{ unstable set} for nonautonomous random dynamical systems.
\begin{definition}
	Let $(\psi,\Theta)$ be a nonautonomous random dynamical system and 
	$\xi^*:\mathbb{R}\times \Omega\to X$ be a random hyperbolic solution of 
	$(\psi,\Theta)$.
	The \textbf{unstable set} of $\xi^*$ is the family
	\begin{equation*}
		W^u(\xi^*)=\{W^u(\xi^*;\omega_\tau): \omega_{\tau}\in \mathbb{R}\times \Omega\},
	\end{equation*}
	where, for each $\omega_{\tau}$, 
	$W^u(\xi^*;\omega_{\tau})$ is the unstable set of 
	the hyperbolic solution $t\mapsto \xi^*(\Theta_t\omega_{\tau})$ of the evolution process
	$\Psi_{\omega_{\tau}}=\{\psi(t-s,\Theta_{s}\omega_{\tau}): t\geq s\}$.
	{\bf The section of $W^u(\xi^*;\omega_\tau)$ at time $t\in \mathbb{R}$} is denoted by
	\begin{eqnarray*}
		W^u(\xi^*;\omega_\tau)(t)=\{z\in X: (t,z)\in W^u(\xi^*;\omega_\tau) \}.
	\end{eqnarray*}
	Let $\delta:\mathbb{R}\times \Omega\to (0,+\infty)$ be a $\Theta$-invariant map, 
	a {\bf local unstable set} is a family
	$W^{u,\delta}(\xi^*)=\{W^{u,
	\delta}(\xi^*;\omega_\tau): \omega_{\tau}\in \mathbb{R}\times \Omega\}$, where 
	\begin{eqnarray*}
		W^{u,\delta}(\xi^*;\omega_\tau)=\bigg\{(t,z)\in \mathbb{R}\times X: 	\hbox{ there is a global solution } \zeta \hbox{ of } 
		\Psi_{\omega_\tau} \hbox{ such that }\\
		\zeta(t)=z, \ \ \|\zeta(s)-\xi^*(\Theta_s\omega_\tau)\|_X\leq \delta(\omega_\tau), \ \forall s\leq t, \\
		\hbox{ and }
		\lim_{s\to -\infty} \|\zeta(s)-\xi^*(\Theta_s\omega_\tau)\|_X=0
		\bigg\},
	\end{eqnarray*}
	and {\bf the section of $W^{u,\delta}(\xi^*;\omega_\tau)$ at time $t$} is defined by 
	\begin{eqnarray*}
		W^{u,\delta}(\xi^*;\omega_\tau)(t)=\{z\in X: (t,z)\in W^{u,\delta}(\xi^*;\omega_\tau) \}.
	\end{eqnarray*}
\end{definition}

\par This definition can also be seen as an
extension of the linear case, see \cite[Theorem 2.2]{Caraballo-Carvalho-Langa-OliveiraSousa-2021}.


%
\par For the unstable set we have the following proposition.

\begin{proposition}\label{prop-properties-Wu}
	Let $(\psi,\Theta)$ be a nonautonomous random dynamical system and 
	$\xi^*:\mathbb{R}\times \Omega\to X$ be a random hyperbolic solution of 
	$(\psi,\Theta)$.
	\par For each $\omega_{\tau}\in \mathbb{R}\times \Omega$ and $t\in \mathbb{R}$,
	\begin{equation}\label{eq-relation-Wu(0,theta_tomegatau)-with-Wu(t,omegatau)}
	W^u(\xi^*;\omega_\tau)(t)=W^u(\xi^*;\Theta_{t}\omega_\tau)(0).
	\end{equation}
	Moreover, if $(\psi,\Theta)$ has a nonautonomous random attractor 
	$\{\mathcal{A}(\omega_\tau): \omega_{\tau}\in \mathbb{R}\}$ and 
	$\xi^*$ is bounded, then
	\begin{equation}
	W^u(\xi^*;\omega_\tau)(0)\subset \mathcal{A}(\omega_{\tau}),
	\ \forall \, \omega_{\tau}.
	\end{equation}
\end{proposition}
\begin{proof}
	\par First we prove \eqref{eq-relation-Wu(0,theta_tomegatau)-with-Wu(t,omegatau)}.
	Let $z\in W^u(\xi^*;\omega_\tau)(t)$, then there exists 
	a global solution $\zeta:\mathbb{R}\to X$ of $\Psi_{\omega_{\tau}}$ such that
	$\zeta (t)=z$ and $\|\zeta(s)-\xi^*(\Theta_s\omega_{\tau})\|_X\stackrel{s\to-\infty}{\longrightarrow} 0$.
	Define, $\tilde{\zeta}(s)=\zeta(t+s)$, $s\in \mathbb{R}$, thus
	$\tilde{\zeta}$ is a global solution for
	$\Psi_{\Theta_t\omega_{\tau}}$ such that
	$\tilde{\zeta}(0)=z$ and 
	\begin{equation*}
		\|\tilde{\zeta}(s)-\xi^*(\Theta_s\Theta_{t}\omega_{\tau})\|_X
		=
		\|\zeta(s+t)-\xi^*(\Theta_{s+t}\omega_{\tau})\|_X
		\stackrel{s\to-\infty}{\longrightarrow} 0.
	\end{equation*}
	Therefore, $z\in W^u(\xi^*,\Theta_t\omega_\tau)(0)$. 
	By similar arguments, we see that
	\begin{equation*}
		W^u(\xi^*,\Theta_t\omega_\tau)(0)\subset W^u(\xi^*,\omega_\tau)(t),
	\end{equation*}
	which concludes the proof of \eqref{eq-relation-Wu(0,theta_tomegatau)-with-Wu(t,omegatau)}.
	\par For the second claim, let
	$z\in W^u(\xi^*;\omega_\tau)(0)$, then there exists a 
	global solution $\zeta:\mathbb{R}\to X$ of $\Psi_{\omega_{\tau}}$ such that
	$\zeta (0)=z$ and $\|\zeta(s)-\xi^*(\Theta_s\omega_{\tau})\|_X\to 0$
	as $s\to -\infty$. 
	Since $\{\xi^*(\Theta_t\omega_{\tau}): t\in (-\infty,0]\}$ is bounded, the set 
	$B=\zeta((-\infty,0])$ is bounded and therefore
	\begin{equation}
	\lim_{s\to -\infty}
	dist_H(\psi(-s,\Theta_s\omega_{\tau})B,\mathcal{A}(\omega_{\tau}) )=0,
	\end{equation} 
	then 
	$d(z,\mathcal{A}(\omega_{\tau}))=0$   
	and 
	$z\in \mathcal{A}(\omega_{\tau})$ for each $\omega_{\tau}$.
	The proof is complete.
\end{proof}
	\par Proposition \ref{prop-properties-Wu} implies that the attractor contains all the unstable sets of hyperbolic solutions. Later, in Section \ref{sec:top-struc-stab}, we will give conditions under which the attractor is equal to the union of these unstable sets. 
	Next, we prove that the local unstable sets for these hyperbolic solutions are given as a graph, following the same line of arguments presented in \cite{Carvalho-Langa-2}. 
	In fact, 
	if $\xi_\eta^*$ is a random hyperbolic solution of $\psi_\eta$, 
	we will show that
the elements in $W^{u,\delta}_\eta(\xi_\eta^*;\omega_{\tau})$ will be those of the form 
\begin{equation*}
	(t,\xi^*_\eta(\Theta_t \omega_\tau)+\Pi^u_\eta(\Theta_t \omega_\tau)z+\Sigma^u(\omega_\tau)(t, \Pi^u_\eta(\Theta_t \omega_\tau)z))\in \mathbb{R}\times X, \hbox{ and } \|z\|_X\leq \delta(\omega_\tau),
\end{equation*}
where $\delta:\mathbb{R}\times \Omega\to (0,+\infty)$ is a $\Theta$-invariant map, for some Lipschitz map $\Sigma^u$. Moreover, we will obtain that as $\eta\to 0$ these local unstable sets ``converges'' to the unstable sets of the autonomous problem \eqref{eq-autonomous-semilinear-ODE}.
\par Let $\omega_{\tau}$ be fixed, $\eta\in (0,\eta_0(\omega_{\tau})]$, and
 $t\mapsto \xi_\eta^*(\Theta_t\omega_{\tau})$ a hyperbolic solution of 
$\{\psi_{\eta}(t-s,\Theta_s\omega_{\tau}): t\geq s \}$ obtained by Theorem 
\ref{th-existence-hyperbolic-solutions-complete}. Then, the change of variables $z(t)=y(t)-\xi_{\eta}^*(\Theta_t\omega_\tau)$
allows us to concentrate on the existence of invariant sets of global hyperbolic solutions around the zero solution for 
\begin{equation}\label{eq-equation-around-zero-equilibria}
\dot{z}=Az+B_{\eta}(\Theta_t\omega_{\tau})z+ h_\eta(\Theta_t\omega_\tau,z),  \ \  z(s)=z_0\in X,
\end{equation}
where $A=B+f_0^\prime(y^*)$, $B_{\eta}(\omega_{\tau})=(f_{\eta})_x(\omega_{\tau},\xi_{\eta}^*(\omega_{\tau}))-f_0^\prime(y^*)$ and 
\begin{equation*}
	\begin{split}
		h_\eta(\Theta_t\omega_\tau,z):=&f_\eta(\Theta_t\omega_\tau,\xi_{\eta}^*(\Theta_t\omega_\tau)+z)-f_\eta(\Theta_t\omega_\tau,\xi_{\eta}^*(\Theta_t\omega_\tau)) \\
		&- (f_\eta)_z(\Theta_t\omega_\tau,\xi_{\eta}^*(\Theta_t\omega_\tau))z.
		\end{split}
\end{equation*} 
Thus $z=0$ is a globally defined bounded solution for \eqref{eq-equation-around-zero-equilibria} where $h_\eta(\omega_{\tau},\cdot):X\to X$ differentiable with
$h_\eta(\omega_{\tau},0)=0$, $(h_\eta)_x(\omega_\tau,0)=0\in\mathcal{L}(X)$, for all 
$\eta\in [0,\eta_{0}(\omega_{\tau})]$.
Furthermore, \eqref{eq-condition-f_eta-f_0} implies that
\begin{equation}\label{eq-condition-h_eta-h_0}
\lim_{\eta\to 0}\sup_{(t,x)\in\mathbb{R}\times B(0,r)}\Big\{\|h_\eta(\Theta_t\omega_\tau,x)-h_0(x)\|_X+
\|(h_\eta)_x(\Theta_t\omega_\tau,x)-h_0^\prime(x)\|_{\mathcal{L}(X)}\Big\}= 0,
\end{equation}
for all $r> 0$ and $\omega_{\tau}\in \mathbb{R}\times\Omega$.
\par We recall that, $\eta_0$ is chosen such that 
the linear evolution process
$\{\varphi_\eta(t-s,\Theta_{s}\omega_{\tau}): t\geq s\}$,
 given by 
\begin{equation}\label{eq-linear-nonautonomous-random-DS-epsilon}
\varphi_\eta(t-s,\Theta_{s}\omega_{\tau})= e^{A(t-s)}+\int_{s}^{t} e^{A(t-r)}B_\eta(\Theta_r\omega_\tau) \varphi_\eta(r-s,\Theta_s\omega_{\tau})\, dr, \ t\geq s,
\end{equation}
admits an exponential dichotomy with 
bound 
$M_\eta$, exponent $\alpha_\eta$ and family of projections
$\{\Pi^u_\eta(t): t\in \mathbb{R} \}$, for every $\eta\in (0,\eta_{0}(\omega_{\tau})]$, see \cite[Theorem 3.9]{Caraballo-Carvalho-Langa-OliveiraSousa-2021} for details. 
Moreover, for each $\Theta$-invariant function 
$\bar{\eta}: \mathbb{R}\times \Omega \to [0,1]$, with 
$\bar{\eta}(\omega_{\tau})\leq \eta_{0}(\omega_{\tau})$, the co-cycle
$(\varphi_{\bar{\eta}},\Theta)$ admits an exponential dichotomy 
with 
bound 
$M_{\bar{\eta}}$, exponent $\alpha_{\bar{\eta}}$ and family of projections
$\{\Pi^u_{\bar{\eta}}(\omega_{\tau}):\omega_\tau\in \mathbb{R} \times \Omega \}$.

\par If $z$ is a solution of 
\eqref{eq-equation-around-zero-equilibria}
we write $z^u(t)=\Pi^u_\eta(t)z(t)$ and
$z^s(t)=\Pi^s_\eta(t)z(t)$, $t\in \mathbb{R}$, 
where $\Pi^u_\eta(t)=Id_X-\Pi_\eta^s(t)$, $t\in \mathbb{R}$.
 Then $z^u$ and $z^s$ are the solutions of
\begin{equation}\label{eq-equation-around-zero-equilibria-projected}
\begin{split}
&\dot{z}^u=A_{\eta}(\Theta_t\omega_{\tau})z^u+ h_\eta^u(\Theta_t\omega_\tau,z^u(t)+z^s(t)), \\
&\dot{z}^s=A_{\eta}(\Theta_t\omega_{\tau})z^s+ h_\eta^s(\Theta_t\omega_\tau,z^u(t)+z^s(t)),
\end{split}
\end{equation}
where $A_{\eta}(\omega_{\tau})=A+B_{\eta}(\omega_{\tau})$, and $h_\eta^k(\omega_\tau,\cdot)=\Pi_\eta^k(\omega_{\tau})
h_\eta(\omega_\tau,\cdot), \ \ k=u,s$.

\par Since, for each $\omega_{\tau}$ fixed, $h_\eta^k(\Theta_t\omega_{\tau},0)=0$, with $(h_\eta^k)_x(\Theta_t\omega_{\tau},0)=0$ and $h_\eta^k$ are continuous differentiable in $X$, uniformly with respect to $t$, we obtain that given $\rho>0$ there exists $\delta_0(\omega_{\tau})>0$
such that 
if $\|z\|_X,\|\tilde{z}\|_X\leq \delta_0(\omega_{\tau})$ then  
\begin{equation}\label{eq-lipschitiz-condition-on-h^k}
\begin{split}
&\sup_{t\in \mathbb{R}}\|h^k_\eta(\Theta_t\omega_{\tau},z)\|_X\leq \rho, \\
&\sup_{t\in \mathbb{R}}\|h^k_\eta(\Theta_t\omega_{\tau},z)-h^k_\eta(\Theta_t\omega_{\tau},\tilde{z})\|\leq \rho \|z-\tilde{z}\|_X, \ \ k=s,u.
\end{split}
\end{equation}
Note that, from \eqref{eq-condition-h_eta-h_0}, it is possible to choose $\delta_0:\mathbb{R}\times\Omega\to (0,+\infty)$ as a $\Theta$-invariant function. 
This is one of the main differences to the deterministic case and to work with the $\Theta$-invariance is the key to our further results.
\begin{remark}\label{remark-cutoff-out-ngh}
	For each $\omega_\tau$ fixed, it is possible to extend $h^u_\eta(\omega_{\tau},\cdot),h^s_\eta(\omega_\tau,\cdot)$ outside the ball of radius $\delta_0(\omega_{\tau})$ such that this extension satisfies both conditions in \eqref{eq-lipschitiz-condition-on-h^k} for all $z,\tilde{z}\in X$, see \cite{Carvalho-Langa-2}.
	Therefore, 
	we obtain the existence and continuity of unstable and stable set, as a graph, for $h_\eta^u$ and $h_\eta^s$ satisfying \eqref{eq-lipschitiz-condition-on-h^k}, for all $z,\tilde{z}\in X$, 
	then, using a localization procedure, we conclude existence and continuity of local unstable sets, as a graph, for the case when $h_\eta^k$ satisfies \eqref{eq-lipschitiz-condition-on-h^k} in the ball of radius $\delta(\omega_{\tau})$, for each $\omega_{\tau}\in \mathbb{R}\times \Omega $.
\end{remark}

\par  
%

\par Assuming that \eqref{eq-lipschitiz-condition-on-h^k} holds for all $z,\tilde{z}\in X$, we will obtain that, for all suitably small $\rho$, the unstable sets are graphs of Lipschitz maps in the class defined next.
Given $L>0$ and a family of projections $\{\Pi^u(s): s\in \mathbb{R}\}$. 
Denote by $\mathcal{LB}(L)$ a complete metric space of all bounded and globally Lipschitz continuous functions 
$\Sigma:\mathbb{R}\times X\to X$ such that 
$\mathbb{R}\times X\ni (s,z)\mapsto \Sigma(s,z):=\Sigma(s,\Pi^u(s)z)\in \Pi^s(s)X$ 
and
\begin{equation}
\begin{aligned}
&\sup \big\{\|\Sigma(s,\Pi^u(s)z)\|_X; (s,z) \in \mathbb{R}\times X \big\}\leq L,\\
&\|\Sigma(s,\Pi^u(s)z)- \Sigma(s,\Pi^u(s)\tilde{z})\|_X\leq L\|\Pi^u(s)z- \Pi^u(s)\tilde{z}\|_X,
\end{aligned}
\end{equation}
with distance between $\Sigma, \widetilde{\Sigma}\in \mathcal{LB}(L)$ given by
\begin{equation}
|\!|\!|\Sigma-\tilde{\Sigma}|\!|\!|:=\sup_{(t,z)\in \mathbb{R}\times X} \|\Sigma(t,z)- \tilde{\Sigma}(t,z)\|_X.
\end{equation}

\begin{theorem}\label{th-existence-continuity-unstable-set}
	Let $\omega_{\tau}\in \mathbb{R}\times \Omega$ be fixed, and $\eta\in [0,\eta_0(\omega_{\tau})]$.
	Suppose that $\rho>0$ is suitable small such that there is 
	$L=L(\rho,\alpha_\eta,M_\eta)>0$ satisfying
	\begin{equation}\label{eq-hypotheses-existence-unstable}
	\begin{aligned}
	&\frac{\rho M_\eta}{\alpha_\eta}\leq L, \ \ \ 
	\frac{\rho M_\eta}{\alpha_\eta}(1+L)<1\\
	&\frac{\rho M_\eta^2 (1+L)}{\alpha_\eta-\rho M_\eta 
	(1+ L)}
	\leq L, \\ 
	&\rho M_\eta +  
	\frac{\rho^2 M_\eta^2 (1+L)(1+M_\eta)}{2\alpha_\eta-\rho M_\eta 
	(1+ L)}< \frac{\alpha_\eta}{2}.
	\end{aligned}
	\end{equation}
	Then, for each 
	$\omega_{\tau}\in \mathbb{R}\times\Omega$ fixed and 
	$\eta\in (0,\eta_{0}(\omega_{\tau})]$, 
	there exists 
	$\Sigma^u_\eta=\Sigma_{\eta,\omega_\tau}^u\in \mathcal{LB}(L)$, such that the unstable set of the zero solution of \eqref{eq-equation-around-zero-equilibria}
	is given by
	\begin{equation}
	W^u_\eta(0)=\{(s,z)\in \mathbb{R}\times X: z=\Pi^u_\eta(s)z+\Sigma_\eta^u(s,\Pi^u_\eta(s)z)\},
	\end{equation}
	and, for any $r> 0$ and $s\in \mathbb{R}$,
	\begin{equation}\label{eq-hypothesis-to-obtain-continuity-unstable}
	\sup_{t\leq s}\sup_{\|z\|_X\leq r} \{\|\Pi^u_\eta(t) z-\Pi^u_0 z\|_X+
	\|\Sigma^{u}_\eta(t,\Pi^u_\eta(t)z) - \Sigma^{u}_0(\Pi^u_0 z)\|_X\}\stackrel{\eta\to 0}{\to } 0.
	\end{equation}
	\par Furthermore, 
	if $\zeta(t)=\zeta^u(t)+\zeta^s(t)$, where $\zeta^k(t)=\Pi_\eta^k(t)\zeta(t)$, for $k=u,s$, is a backward-bounded global solution of \eqref{eq-equation-around-zero-equilibria}, then there is 
	$\gamma>0$ such that,
	\begin{equation}\label{eq-expontial-attraction-backwards}
	\|\zeta^s(t)-\Sigma_\eta^u(t,\zeta^u(t))\|_X\leq 
	M_\eta
	e^{-\gamma(t-s)}\|\zeta^s(s)-\Sigma_\eta^u(s,\zeta^u(s))\|_X, \ t\geq s.
	\end{equation}
\end{theorem}
\par Theorem \ref{th-existence-continuity-unstable-set} follows directly from 
\cite[Theorem 3.1]{Carvalho-Langa-2}.

\par From Theorems \ref{th-existence-hyperbolic-solutions-complete} and \ref{th-existence-continuity-unstable-set}, we can obtain the existence and continuity of local unstable sets. 
%

\begin{theorem}[Existence and continuity of local unstable set]
	\label{th-existence-continuity-local-unstable-set}
	\par Let $\eta\in [0,1]$, and
	$h_\eta:\mathbb{R}\times \Omega\times X\to X$ by such that 
	for each $\omega_{\tau}$, the mapping
	$z\mapsto h_\eta(\omega_{\tau},z)$ is continuously differentiable. 
	Consider 
	\begin{equation}\label{eq-theorem-existence-continuity-local-Wu}
	\dot{z}=A_\eta(\Theta_t\omega_{\tau})z+
	h_\eta(\Theta_{t}\omega_\tau,z), \ \ \omega_{\tau} \in \mathbb{R}\times \Omega.
	\end{equation}
	Assume that
	$h_\eta(\omega_{\tau},0)=0$, $(h_\eta)_x(\omega_\tau,0)=0\in\mathcal{L}(X)$,
	$h_0:X\to X$, $A_0(\Theta_t\omega_{\tau})=A$, 
	$\{h_\eta\}_{\eta\in [0,1]}$ satisfies
	\eqref{eq-condition-h_eta-h_0}, and that
	$z_0^*=0$ is a hyperbolic solution of \eqref{eq-theorem-existence-continuity-local-Wu} for $\eta=0$.
	Then given $\epsilon_0>0$ suitable small, the following hold:
	\begin{enumerate}
		\item There exist a $\Theta$-invariant function $\eta_{0}:\mathbb{R}\times \Omega\to [0,1]$ such that $z_\eta^*=0$ is a hyperbolic solution of 
		\eqref{eq-theorem-existence-continuity-local-Wu},
		for each $\eta\in (0,\eta_0(\omega_{\tau})]$.
		In particular, 
		the linear evolution process $\{\varphi_\eta(t-s,\Theta_s\omega_{\tau}): t\geq s\}$,
		associated to the linear part of 
		\eqref{eq-theorem-existence-continuity-local-Wu} 
		(corresponding to the linearization of 
		$\{\psi_{\eta}(t-s,\Theta_s\omega_{\tau}): t\geq s \}$ around
		 $\xi_\eta^*(\Theta_{t}\omega_{\tau})$), 
		 admits an
		 exponential dichotomy with family of projections $\{\Pi_\eta^u(s): s\in \mathbb{R} \}$.
		\item The families of projections
		$\Pi^u_\eta=\{\Pi^u_\eta(s): s\in \mathbb{R}\}$, $\eta\in  (0,\eta_0(\omega_{\tau})]$ 
		satisfy
		\begin{equation}
		\lim_{\eta\to 0} \sup_{t\in \mathbb{R}} \|\Pi^u_\eta(t) -\Pi^u_0\|_{\mathcal{L}(X)}=0.
		\end{equation}
		\item
		There exist $\Theta$-invariant function
		$\delta_0:\mathbb{R}\times \Omega\to (0,+\infty)$
		 (independent of $\eta$) 
		 such that 
		  for each $\omega_{\tau}$ and 
		 $\eta\in [0,\eta_0(\omega_\tau)]$, and
		 a map
		 \begin{equation}
		 \mathbb{R}\times B_X(0,\delta_0(\omega_{\tau}))\ni (s,z)\mapsto \Sigma^u_\eta(s,z):=\Sigma^{u}_\eta(s, \Pi^u_\eta(s)z),
		 \end{equation}
		 with the property:
		 given $\delta\in (0,\delta_0(\omega_{\tau}))$, there exists $0<\delta''<\delta'<\delta$, 
		\begin{equation}\label{eq-local-unstable-as-graph}
		\begin{split}
		&\{\Pi^u_\eta(s)z+ \Sigma^{u}_{\eta}(s,\Pi^u_\eta(s)z) \,:\, \|z\|_X\leq \delta''\} \subset \\
		&W^{u,\delta'}_{\eta}(0)(s)\subset \\
		&\{ \Pi^u_\eta(s)z+ \Sigma^{u}_{\eta}(s,\Pi^u_\eta(s)z) \,:\, \|z\|_X\leq \delta\}.
		\end{split}
		\end{equation}
		\item For each $\omega_{\tau}$ fixed, the family of graphs of the maps $\{\Sigma_\eta\}_{\eta\in (0,\eta_{0}(\omega_{\tau})]}$ behaves continuously at $\eta=0$:
		\begin{equation}
		\sup_{t\leq s}\sup_{\|z\|\leq \delta_0(\omega_\tau)} \{\|\Pi^u_\eta(t)-\Pi^u_0\|_{\mathcal{L}(X)}
		+
		\|\Sigma^{u}_\eta(t,\Pi^u_\eta(t)z) - \Sigma^{u}_0(\Pi^u_0 z)\|_X\}\stackrel{\eta\to 0}{\to } 0, \ \forall \, s\in \mathbb{R}.
		\end{equation}
	\end{enumerate}
\end{theorem}

\begin{proof}
	\par Item (1) is a corollary of Theorem \ref{th-existence-hyperbolic-solutions-complete} and
	Item (2) follows from the continuous dependence of projections, in the sense of 
	\cite[Theorem 7.9]{Carvalho-Langa-Robison-book} for evolution processes (see also \cite[Theorem 2.23]{Caraballo-Carvalho-Langa-OliveiraSousa-2021} for nonautonomous random dynamical systems).
	\par By hypotheses, let $\rho>0$ be such that there is $L$ satisfying
	\eqref{eq-hypotheses-existence-unstable}, 
	then there exists 
	$\delta_0(\omega_{\tau})$ such that 
	\eqref{eq-lipschitiz-condition-on-h^k} is satisfied for
	$z,\bar{z}\in B_X(0,\delta_0(\omega_\tau))$.
	\par According to
	Remark \ref{remark-cutoff-out-ngh} and Theorem \ref{th-existence-continuity-unstable-set}, by a cut-off procedure, 
	we obtain the desired function $\Sigma_\eta^u:\mathbb{R}\times B_X(0,\delta_0(\omega_\tau))\to X$, for each 
	$\eta\in (0,\eta_{0}(\omega_{\tau})]$.
	\par Thus, we only need to check \eqref{eq-local-unstable-as-graph}.
	We claim that given $\delta\in (0,\delta_0(\omega_\tau))$, 
	there exists
	$\delta^\prime<\delta$
	such that 
	any global solution 
	$\zeta:\mathbb{R}\to X$ of 
	$\{\psi_\eta(t-s,\Theta_s\omega_{\tau}): t\geq s\}$ on the unstable set
	such that
	$\|\zeta(s)\|\leq \delta^\prime$ must satisfy
	$\|\zeta(t)\|\leq \delta$, for $t\leq s$.
	\par Indeed, from \eqref{eq-equation-around-zero-equilibria-projected}, $\zeta^u(t)=\Pi_\eta^u(t)\zeta(t)$ satisfies 
	\begin{equation*}
	\begin{aligned}
	\zeta^u(t)&=\varphi_\eta(t-s,\Theta_s\omega_{\tau})\Pi^u_\eta(s)\zeta_0  \\
	& \ \ \ \ + \ \int_s^t \varphi_\eta(t-r,\Theta_r\omega_{\tau})\Pi^u_\eta(r)h_\eta^u(\Theta_r\omega_{\tau},\zeta^u(r)+
	\Sigma_\eta^u (r,\zeta^u(r)))dr,
	\ t\leq s.
	\end{aligned}
	\end{equation*}
	Since $\{\varphi_\eta(t-s,\Theta_s\omega_{\tau}): t\geq s\}$ admits an exponential dichotomy, due to
	Gr\"onwall's inequality, we obtain
	\begin{equation*}
	\|\zeta^u(t)\|_X\leq M_\eta e^{(\alpha_\eta-\rho M_\eta(1+L) )(t-s)}\|\zeta^u(s)\|_X, \ \ t\leq s.
	\end{equation*}
	Also, since
	$\|\Sigma_\eta^u(t,\zeta^u(t))\|_X\leq L\|\zeta^u(t)\|_X$, $t\in \mathbb{R}$, we have that
	\begin{equation}\label{eq-boundedness-for-zeta-in-W^u}
	\|\zeta(t)\|_X\leq M_\eta^2(1+L)e^{(\alpha_\eta-\rho M_\eta(1+L) )(t-s)}\|\zeta(s)\|_X, \ \ t\leq s.
	\end{equation}
	Then, taking $\delta'=\delta / M_\eta^2(1+L)$,
	 we see that
	\begin{equation}
	W^{u,\delta'}_{\eta}(0)(s)\subset
	\{ \Pi^u_\eta(s)z+ \Sigma^{u}_{\eta}(s,\Pi^u_\eta(s)z) \,:\, \|z\|_X\leq \delta\}.
	\end{equation}
	Finally, by the above argument, 
	we also conclude that there exists $\delta''\in (0,\delta')$ such that 
	\begin{equation}
	\{\Pi^u_\eta(s)z+ \Sigma^{u}_{\eta}(s,\Pi^u_\eta(s)z) \,:\,  \|z\|_X\leq \delta''\} \subset \\
	W^{u,\delta'}_{\eta}(0)(s).
	\end{equation}
	The proof is complete.
	\end{proof}

\begin{remark}
	We observe that, as in Theorem \ref{th-existence-hyperbolic-solutions-complete}[Item (2)], using $\Theta$-invariant functions $\bar{\eta}:\mathbb{R}\times \Omega\to (0,1]$ it is possible to conclude existence of local unstable manifolds of the random hyperbolic solutions $\xi_{\bar{\eta}}^*$ for the nonautonomous random dynamical systems $\psi_{\bar{\eta}}$.
\end{remark}

\par We reinforce that these results on the existence and continuity of local unstable sets are the key to obtain lower semicontinuity and topological structural stability, as we will see in the following sections.

\begin{remark}
	We can obtain similar results concerning the existence and continuity of local stable sets following similar arguments to those presented here and \cite{Carvalho-Langa-2} for the deterministic case. 
\end{remark}

\section{Continuity of nonautonomous random attractors}
\label{sec:continuity-of-NR-Attractors}

\par In this section, we prove the continuity of attractors in the situation that the perturbed system is nonautonomous random whereas the limiting is an autonomous dynamical system which has an attractor given as union of unstable sets of hyperbolic equilibria. 

\par First, we recall the definition of continuity of sets in Banach space $X$.	
Let $\{\mathcal{A}_\eta\}_{ \eta\in [0,1]}$ be a family of subsets of a Banach space $X$. We say that 
	$\{\mathcal{A}_\eta\}_{ \eta\in [0,1]}$ is 
	\begin{itemize}
		\item[\textbf{(1)}] \textbf{Upper semicontinuous} at $\eta=0$ if 
		$\lim_{\eta\to 0} dist_H(A_\eta,A_0)=0$.
		\item[\textbf{(2)}] \textbf{Lower semicontinuous} at $\eta=0$ if 
		$\lim_{\eta\to 0} dist_H(A_0,A_\eta)=0$.
		\item[\textbf{(3)}] \textbf{Continuous} at $\eta=0$ if it is upper and lower semicontinuous at $\eta =0$.
	\end{itemize}
	Let $\Lambda$ be a nonempty set. 
	We say that $\{\mathcal{A}_\eta(\lambda): \lambda\in \Lambda\}_{ \eta\in [0,1]}$ is \textbf{upper (lower) semicontinuous at $\eta=0$} if 
	$\{\mathcal{A}_\eta(\lambda)\}_{ \eta\in [0,1]}$ is upper (lower) semicontinuous at $\eta=0$, for each $\lambda\in \Lambda$, see \cite[Chapter 3]{Carvalho-Langa-Robison-book}.

\par Now, we present a result on the continuity of attractors, as a consequence of a careful study of their internal structure, presented in the previews sections.

\begin{theorem}[Continuity of nonautonomous random attractors]
	\label{th-continuity-nonautonomous-random-attractors}
	Let $\mathcal{T}_0=\{T_0(t):t\geq 0\}$ be the semigroup associated to \eqref{eq-autonomous-semilinear-ODE} and $(\psi_\eta,\Theta)$ be the nonautonomous dynamical systems associated to \eqref{eq-nonautonomous-semilinear-randomODE}, and assume that condition \eqref{eq-condition-f_eta-f_0} is satisfied.
	Additionally, suppose that 
	\begin{itemize}
		\item[\textbf{(a)}] For each $\eta\in [0,1]$, the co-cycle $(\psi_\eta,\Theta)$ has a nonautonomous random attractor $\{\mathcal{A}_\eta(\omega_\tau): \omega_\tau\in \mathbb{R}\times \Omega\}$
		and 
		\begin{equation*}
			\overline{\bigcup_{t\in \mathbb{R}} \bigcup_{\eta\in [0,1]} \mathcal{A}_\eta(\Theta_t\omega_\tau)} \hbox{ is compact,} \ \ \forall \, \omega_{\tau}\in \mathbb{R}\times \Omega;
		\end{equation*}
		\item[\textbf{(b)}] $\mathcal{T}_0=\{T_0(t):t\geq 0\}$ is a semigroup with global attractor
		given by 
		\begin{equation}\label{eq-grad-strucuture-limit}
		\mathcal{A}_0=\bigcup_{j=1}^p W^u(y_j^*),
		\end{equation} 
		for which all the equilibria $\{y^*_j: 1\leq j\leq p\}$ are hyperbolic.
	\end{itemize}
	Then given $\epsilon_0>0$ suitable small, there exists a $\Theta$-invariant function $\eta_{0}:\mathbb{R}\times \Omega\to (0,1]$ such that, for each $\omega_{\tau}$ fixed, the following hold:
	\begin{enumerate}
		\item For any $\eta \in (0,\eta_0(\omega_{\tau})]$ and $j\in \{1,\cdots,p\}$, there exists a hyperbolic solution $\xi_{j,\eta}^*$ of $\{\psi_\eta(t-s,\Theta_s\omega_{\tau}): t\geq s\}$ 
		with
		\begin{equation}
		\sup_{j} \sup_{t\in \mathbb{R}} \|\xi_{j,\eta}^*(\Theta_{t}\omega_\tau)-y_j^*\|_X<\epsilon_0,
		\end{equation}
		where 
		the linearized associated evolution process  
		admits an exponential dichotomy with family of projections $\{\Pi_{j,\eta}^u(s): s\in \mathbb{R}\}$.
		\item There exists
		$\delta_0(\omega_{\tau})>0$,
		where $\delta_0$ is $\Theta$-invariant and independent of $\eta$, 
		such that 
		for each $\omega_{\tau}$, $j\in \{1,\cdots,p\}$, and
		$\eta\in [0,\eta_0(\omega_\tau)]$, 
		there exists 
		a map
		\begin{equation}
		\mathbb{R}\times B_X(0,\delta_0(\omega_{\tau}))\ni (s,z)\mapsto \Sigma^u_{j,\eta}(s,z):=\Sigma^{u}_{j,\eta}(s, \Pi^u_{j,\eta}(s)z),
		\end{equation}
		with the property:
		given $\delta\in (0,\delta_0(\omega_{\tau}))$, there exists $0<\delta''<\delta'<\delta$, 
		\begin{equation}
		\begin{split}
		&\{\xi_{j,\eta}^*(s)+\Pi^u_{j,\eta}(s)z+ \Sigma^{u}_{j,\eta}(s,\Pi^u_{j,\eta}(s)z) \,:\,  \|z\|_X\leq \delta''\} \subset \\
		&W^{u,\delta'}_{j,\eta}(\xi_{j,\eta}^*)(s)\subset \\
		&\{\xi_{j,\eta}^*(s)+ \Pi^u_{j,\eta}(s)z+ \Sigma^{u}_{j,\eta}(s,\Pi^u_{j,\eta}(s)z) \,:\, \|z\|_X\leq \delta\};
		\end{split}
		\end{equation}
		 \item The family of graphs of $\{\Sigma_{j,\eta}^u\}_{\eta\in [0,\eta_{0}(\omega_{\tau})]}$ 
		is continuous at $\eta=0$ as in Theorem \ref{th-existence-continuity-local-unstable-set}[Item (4)], for each $j\in \{1,\cdots,p\}$.
		\item For each $\omega_{\tau}$, the family of pullback attractors
		$\{\mathcal{A}_\eta(\Theta_t\omega_{\tau}): t\in \mathbb{R}\}_{\eta\in [0,\eta_{0}(\omega_{\tau})]}$ is continuous at $\eta=0$.
	\end{enumerate}
	In particular, we have continuity of nonautonomous random attractors in the following sense: 
	given $\epsilon>0$, there exists a $\Theta$-invariant function
	$\eta_{\epsilon}\leq \eta_0$ such that,  
	for every $\Theta$-invariant function $\bar{\eta}$, with
	$\bar{\eta}\leq \eta_{\epsilon}$, we have
	\begin{equation}
	\sup_{t\in \mathbb{R}}d_H(\mathcal{A}_{\bar{\eta}}(\Theta_t\omega_{\tau}),\mathcal{A}_0 )<\epsilon, \ \ 
	\forall \, \omega_{\tau}\in \mathbb{R}\times \Omega,
	\end{equation}
	where $\{\mathcal{A}_{\bar{\eta}}(\omega_{\tau}): \omega_{\tau}\in \mathbb{R}\times \Omega\}$ is the nonautonomous random attractor of 
	$(\psi_{\bar{\eta}},\Theta)$ and 
	$d_H(A,B)=\max\{dist_H(A,B),dist_H(B,A)\}$, for $A,B\subset X$.
\end{theorem}
\begin{proof}
	\par Note that, items (1)-(3) are consequences of Theorem 
	\ref{th-existence-hyperbolic-solutions-complete} and Theorem \ref{th-existence-continuity-local-unstable-set}, thus to conclude the proof
	we only need to prove Item (4).
	\par Note that, from 
	\eqref{eq-condition-f_eta-f_0}, 
	we are able to prove that  
	\begin{equation}\label{th-continuity-at-epsilon-0-thcontinuityattractors}
	\lim_{\eta\to 0} \sup_{t\in [0,T]} \sup_{s\in \mathbb{R}} 
	\sup_{\|z\|_X\leq r}
	\|\psi_\eta(t,\Theta_s \omega_\tau)z- T_0(t)z\|_X\to 0,
	\end{equation}
	for any $T,r>0$ and $\omega_{\tau}\in \mathbb{R}\times \Omega$.
	\par The proof of upper semicontinuity follows from standard arguments using
	\eqref{th-continuity-at-epsilon-0-thcontinuityattractors} and hypothesis
	\textbf{(a)},
	see \cite[Chapter 3]{Carvalho-Langa-Robison-book}, for pullback attractors, and \cite{Caraballo-Langa,Caraballo-Langa-Robinson-98,Bixiang-Wang-existence-upper} for random attractors. 
	\par Now, we prove lower semicontinuity using a characterization via sequences, see \cite[Lemma 3.2]{Carvalho-Langa-Robison-book}. In fact,
	let $\omega_{\tau}\in \mathbb{R}\times \Omega$, $t\in \mathbb{R}$, and $x_0\in \mathcal{A}_0$, we will show that there exist sequences $\eta_{k}\in (0,\eta_0(\omega_{\tau})]$, with $\eta_k\to 0$, and 
	$x_k\in \mathcal{A}_{\eta_{k}}(\Theta_t\omega_\tau)$ such that 
	$x_k\to x_0$ as $k\to +\infty$. 
	\par Indeed, from \eqref{eq-grad-strucuture-limit}, $x_0\in W^u(y^*_j)$ for some $j\in \{1,\cdots,p\}$.
	By Item (3) of
	Theorem \ref{th-existence-continuity-local-unstable-set}, there exist $0<\delta''<\delta'<\delta_0(\omega_{\tau})$ such that
	\begin{eqnarray}
	& &W^{u,\delta''}_0(y_j^*)\subset \{y_j^*+\Pi^u_{j,0}z+ \Sigma^{u}_{0}(\Pi^u_{j,0}z) \,: \,  \|z\|_X\leq \delta'\}, \hbox{ and } \\
	& &\{\xi_{j,\eta}^*(r)+\Pi^u_{j,\eta}(r)z+ \Sigma^{u}_{j,\eta}(r,\Pi^u_{j,\eta}(r)z) \,: \, \|z\|_X\leq \delta'\}
	\subset W^{u,\delta_0}_{\eta}(\xi_{j,\eta}^*)(r),\label{eq-local-unstable-contains-graph}
	\end{eqnarray}
	for every $r\in \mathbb{R}$ and $\eta\in (0,\eta_{0}(\omega_{\tau})]$.
	Thus there exists a global solution $\zeta:\mathbb{R}\to X$ 
	of $\mathcal{T}_0$ such that
	$\zeta(0)=x_0$ and $\zeta(-s)\in W^{u,\delta''}_0(y_j^*)$, for some $s\geq 0$.
	\par Since
	$\zeta(-s)\in \{y_j^*+\Pi^u_{j,0}z+ \Sigma^{u}_{j,0}(\Pi^u_{j,0}z), \ \  \|z\|_X\leq \delta'\}$, by Theorem \ref{th-existence-continuity-local-unstable-set}[Item (4)],
	there exist $\{\eta_k\} \subset (0,\eta_{0}(\omega_{\tau})]$ and 
	$z_k\in \{\xi^*_{j,\eta_k}(t-s)+\Pi^u_{j,\eta_k}(t-s)z+ \Sigma^{u}_{j,\eta_k}(t-s,\Pi^u_{j,\eta_k}(t-s)z): \  \|z\|_X\leq \delta{'}\}$ with 
	$\eta_k\to 0$ and $z_k\to \zeta(-s)$ as $k\to +\infty$.
	\par By \eqref{eq-local-unstable-contains-graph} and Proposition \ref{prop-properties-Wu}, we see that $x_k=\psi_{\eta_k}(t-(t-s),\Theta_{t-s}\omega_{\tau})z_k\in \mathcal{A}_{\eta_{k}}(\Theta_{t}\omega_{\tau})$, for all $k\in \mathbb{N}$. 
	Then, we use \eqref{th-continuity-at-epsilon-0-thcontinuityattractors} 
	and that $\lim_k z_k= \zeta(-s)$, to guarantee that $\lim_k x_k= x_0$, and the proof is complete.
	\end{proof}

	\begin{remark}
		\par Theorem \ref{th-continuity-nonautonomous-random-attractors} can be extended to the case where the limit is nonautonomous. The key steps for the proof will be again the $\Theta$-invariance of the maps involved. 
		
	\end{remark}

	\begin{remark}\label{remark-colectiv-asympt-compact}
		Alternatively, Assumption \textbf{(a)} can be replaced by the following two conditions:
		\begin{itemize}
			\item[\textbf{(a.1)}] For each $\eta\in [0,1]$, the co-cycle $(\psi_\eta,\Theta)$ has a nonautonomous random attractor $\{\mathcal{A}_\eta(\omega_\tau): \omega_\tau\in \mathbb{R}\times \Omega\}$
			and 
			\begin{equation*}
				\overline{\bigcup_{t\in \mathbb{R}} \bigcup_{\eta\in [0,1]} \mathcal{A}_\eta(\Theta_t\omega_\tau)} \hbox{ is bounded,} \ \ \forall \, \omega_{\tau}\in \mathbb{R}\times \Omega;
			\end{equation*}
			\item[\textbf{(a.2)}] The family $\{\psi_\eta,\Theta\}_{\eta\in [0,1]}$ is \textbf{collectively asymptotic compact} in X, i.e., 
			for all $\omega_{\tau}$, the sequence 
				\begin{equation*}
				\{\psi_{\eta_n}(t_n,\Theta_{-t_n}\omega_{\tau})x_n\} \hbox{ has a convergent subsequence in } X
				\end{equation*}
				whenever $\eta_n\to 0$, $t_n\to +\infty$, and $\{x_n\}$ is a bounded sequence in $X$,
		\end{itemize}
	and Theorem \ref{th-continuity-nonautonomous-random-attractors} will still hold true. This will be the case when applying this result for damped wave equations, see Subsection \ref{subsec-app-dff-equations-wave}.
	\end{remark}

	\begin{remark}\label{remark-uniformly-bounded-noises-continuity}
		\par Theorem \ref{th-continuity-nonautonomous-random-attractors} is not optimal in the sense that we can not obtain the limit
		\begin{equation}\label{eq-continuity-attractors-uniformly-in-omega_tau}
		\sup_{\omega_{\tau}\in \mathbb{R}\times \Omega}
		d_H (\mathcal{A}_\eta(\omega_{\tau}), \mathcal{A}_0)\to 0, \ \ \hbox{ as } \eta \to 0.
		\end{equation}	
		To obtain this conclusion one should
		assume
		\begin{equation}\label{eq-condition-f_eta-f_0-uniformly-on-Omega_R}
		\sup_{\omega_{\tau}\in \mathbb{R}\times \Omega}\sup_{x\in B(0,r)}\Big\{\|f_\eta(\omega_\tau,x)-f_0(x)\|_X+
		\|(f_\eta)_x(\omega_\tau,x)-f_0^\prime(x)\|_{\mathcal{L}(X)}\Big\}
		\stackrel{\eta\to 0}{\to} 0,
		\end{equation}
		for all $r\geq 0$, instead of \eqref{eq-condition-f_eta-f_0}.
%
		In this case, it is possible to obtain the conclusions of 
		Theorem \ref{th-continuity-nonautonomous-random-attractors} with
		$\eta_0>0$ and $\delta_0>0$ independent of $\omega_{\tau}$, and therefore to conclude
		\eqref{eq-continuity-attractors-uniformly-in-omega_tau}. Note that this case is similar to the deterministic case, see \cite[Theorem 3.1]{Carvalho-Langa-Robinson}.
		\par However, in the applications to check 
		condition \eqref{eq-condition-f_eta-f_0-uniformly-on-Omega_R} 
		one has to assume that the noise is uniformly bounded as in Remark \ref{remark-example-bounded}, see also \cite{Bobryk-21,Caraballo-Colucci-Cruz-Rapaport-20} for more examples of uniformly bounded noises.
		On the other hand, in Section \ref{sec-app-dff-equations} we provide an example, see Example \ref{example-gradient-finite-dimensional}, where conditions of Theorem \ref{th-continuity-nonautonomous-random-attractors} are checked, but we do not know if 
		its possible to verify \eqref{eq-condition-f_eta-f_0-uniformly-on-Omega_R}.
	\end{remark}

	\par Now, that the study of continuity of attractor is complete, the next step is to prove that the gradient structure is preserved under nonautonomous random perturbation.


\section{Topological structural stability}
\label{sec:top-struc-stab}

\par In this section, we present a result on the topological structural stability of attractors for
nonautonomous random dynamical systems. 
We study co-cycles
$(\psi_\eta,\Theta)$ obtained by nonautonomous random perturbations of 
a gradient semigroup $\{T_0(t): t\geq 0\}$.

\par First, we recall some basic concepts necessary to define \textit{dynamically gradient evolution processes}. Assume that 
	 $\mathcal{S}=\{S(t,s): t\geq s\}$ is an evolution process with a pullback attractor $\{\mathcal{A}(t):t\in \mathbb{R}\}$.    
	\par Let $\widehat{E}=\{E(t): t\in \mathbb{R}\}$ be an invariant 
	family for $\mathcal{S}$. 
	Given a family of open sets
	$\widehat{U}=\{U(t): t \in \mathbb{R}\}$ 
	such that $\widehat{E}\subset \widehat{U}$ (i.e., $E(t)\subset U(t)$, for every $t\in \mathbb{R}$) 
	we say that $\widehat{E}$ is \textbf{the maximal invariant in} $\widehat{U}$ if given an invariant family $\widehat{F}$ in $\widehat{U}$, then $\widehat{F}\subset \widehat{E}$. 
	If there is a $\epsilon_0>0$ such that 
	$\widehat{E}$ is the maximal invariant family in
	$\{O_{\epsilon_0}(E(t)) : t\in \mathbb{R}\}$, we say that $\widehat{E}$ is a \textbf{isolated invariant family}.
	We say that $\{\widehat{E}_1,\cdots, \widehat{E}_p\}$ is a \textbf{disjoint collection of isolated invariant families} if $\widehat{E}_i$ is an isolated invariant family for every 
	$0\leq i\leq p$ and there is ${\epsilon_0}>0$ such that 
	$O_{\epsilon_0}(E_j(t))\cap O_{\epsilon_0}(E_i(t))=\emptyset$,
	for $i\neq j$ and every $t\in \mathbb{R}$.
	A \textbf{homoclinic structure} in $\{\widehat{E}_1,\cdots, \widehat{E}_p\}$
	is a subcollection $\{\widehat{E}_{l_1},\cdots, \widehat{E}_{l_k}\}$, with $k\leq p$, and a 
	set of global solutions
	$\{\zeta_1,\cdots,\zeta_k\}$
	of $(\psi,\Theta)$ in $\mathcal{A}$ which, setting
	$\widehat{E}_{l_k+1}=\widehat{E}_{k_1}$, satisfy 
	\begin{equation}
	\lim_{t\to -\infty} 
	d(\zeta_i(t) , E_{l_i}(t))=0,
	\hbox{ and }
	\lim_{t\to +\infty} 
	d(\zeta_i(t) , E_{l_{i+1}}(t))=0,
	\end{equation}
	for each $1\leq i\leq k$, and there exists a $\epsilon>0$ such that
	\begin{equation}
	\sup_{t\in \mathbb{R}}d(\,\zeta_i(t), \bigcup_{i=1}^p O_\epsilon (E_{l_i}(t)) \, )
	>0, \ \forall \, 1\leq i\leq k, \hbox{ and } t\in \mathbb{R}\times \Omega.
	\end{equation}

\begin{definition}
	We say that $\mathcal{S}=\{S(t,s): t\geq s\}$ is a \textbf{dynamically gradient evolution process} with respect to $\{\widehat{E}_1,\cdots, \widehat{E}_p\}$ if 
	\begin{itemize}
		\item \textbf{(G1)} If $\zeta:\mathbb{R}\to X$ is a global solution of
		$\mathcal{S}$ such that $\zeta(t)\in \mathcal{A}(t)$,
		then there exist $i,j\in \{1,\cdots,p\}$ such that 
		\begin{equation}
		\lim_{t\to -\infty} 
		d(\zeta(t) , E_{i}(t))=0,
		\hbox{ and }
		\lim_{t\to +\infty} 
		d(\zeta(t) , E_{j}(t))=0.
		\end{equation}
		\item  \textbf{(G2)} $\{\widehat{E}_1,\cdots, \widehat{E}_p\}$ does not admit any homoclinic structure.
	\end{itemize}
\end{definition}
\par This notion of dynamically gradient was studied for random dynamical systems in
\cite{Caraballo-Langa-Liu-12,Ju-QI-JWang-18}. For topological structural stability of deterministic autonomous or nonautonomous dynamical systems,
see \cite{Aragao-Costa-2013,Bortolan-Cardoso-etal,Carvalho-Langa-1}.

\par Now, we present our result on the topological structural stability for random dynamical systems.

\begin{theorem}\label{th-topological-structure-stability} 
	Assume that hypotheses of Theorem \ref{th-continuity-nonautonomous-random-attractors} are fulfilled and additionally assume that $\mathcal{T}_0=\{T_0(t-s):t\geq s\}$ is a gradient evolution process with respect to $\{y_1^*,\cdots,y_p^*\}$, 
	where $y^*_j$ is hyperbolic, for every $: 1\leq j\leq p$.
	\par Then, 
	there exists a $\Theta$-invariant function $\eta_{1}:\mathbb{R}\times \Omega\to (0,1)$ 
	such that for each $\omega_{\tau}$ fixed
the evolution process
$\{\psi_\eta(t-s,\Theta_{s}\omega_{\tau}): t\geq s\}$ is
dynamically gradient with respect to $\{\xi_{1,\eta}^*,\cdots,\xi_{p,\eta}^*\}$, $\forall \eta \leq \eta
_1(\omega_{\tau}) $. 
Consequently, 
\begin{equation}
\mathcal{A}_\eta(\Theta_t\omega_\tau)= \bigcup_{j=1}^p W^u_\eta(\xi_{j,\eta}^*;\omega_{\tau})(t),  \forall\, \eta\in [0,\eta_1(\omega_{\tau})].
\end{equation}
\end{theorem}

\begin{proof}
	\par Let $\omega_{\tau}\in \mathbb{R}\times \Omega$ be fixed and
	$\eta\in (0,\eta_0(\omega_{\tau})]$. 
	Let us prove the following claim:
	there exists $\delta'\in (0,\delta_0(\omega_{\tau}))$ such that, if $\zeta_\eta:\mathbb{R}\to X$ is
	a global solution in $\{\mathcal{A}_\eta(\Theta_t\omega_{\tau}):t\in \mathbb{R}\}$ so that
	\begin{equation}\label{eq-condition-top-structure-stab}
	\|\zeta_\eta(t)- \xi_{j,\eta}^*(t)\|_X<\delta', \ \ \forall\, 
	t\leq t_0 \ \ (t\geq t_0), \hbox{ for some } t_0\in \mathbb{R},
	\end{equation}
	then 
	$\|\zeta_\eta(t)-\xi_{j,\eta}^*(t)\|_X\stackrel{t\to-\infty}\longrightarrow 0$ 
	$ \  (\|\zeta_\eta(t)- \xi_{j,\eta}^*(t)\|_X\stackrel{t\to+\infty}\longrightarrow 0)$.
	\par We prove only the backwards case, the proof of the forward case will be similar using the analogous results for the stable sets. 
	First, note that 
	$\tilde{\zeta}(t)=\zeta_\eta(t)- \xi_{j,\eta}^*(t)$, for $t\in \mathbb{R}$, $j\in \{1,\cdots,p\}$, and $\eta\in (0,\eta_{0}(\omega_{\tau})]$, thus we analyze the dynamics around the solution 
	$z=0$ of \eqref{eq-equation-around-zero-equilibria}.
	From Theorem \ref{th-existence-continuity-local-unstable-set}[Item (3)], there exists $0<\delta'<\delta<\delta_0(\omega_{\tau})$ such that
	\begin{equation}\label{eq-graph-proof-top-struc-stabil}
	\{\Pi^u_{j,\eta}(s)z+ \Sigma^{u}_{j,\eta}(s,\Pi^u_{j,\eta}(s)z) \,: \,  \|z\|_X\leq \delta'\} \subset W^{u,\delta}_{\eta}(0)(s), \forall\, s\in \mathbb{R}.
	\end{equation}
	Thus, \eqref{eq-condition-top-structure-stab} implies that $\tilde{\zeta}(t)$ is
	inside the $\delta_0(\omega_{\tau})$-neighborhood for all $t\leq t_0$. 
	\par Hence, from \eqref{eq-expontial-attraction-backwards} applied in the $\delta_0(\omega_{\tau})$-neighborhood of $z=0$, we must have that 
	$\tilde{\zeta}(t_0)\in \{\Pi^u_{j,\eta}(t_0)z+ \Sigma^{u}_{j,\eta}(t_0,\Pi^u_{j,\eta}(t_0)z) \,: \,  \|z\|_X\leq \delta'\}$.
	Therefore, from \eqref{eq-graph-proof-top-struc-stabil}, 
	$\tilde{\zeta}(t_0)\in W^{u,\delta}_{\eta}(0)(t_0)$ and the proof of the claim is complete.
	\par In this way, the proof will be a consequence of  
	\cite[Theorem 8.14]{Bortolan-Carvalho-Langa-book}.
	\end{proof}

\begin{remark}\label{remark-uniformly-bounded-noises-top-struct-stab}
	Note that, if we assume \eqref{eq-condition-f_eta-f_0-uniformly-on-Omega_R} in Theorem \ref{th-continuity-nonautonomous-random-attractors} instead of 
	\eqref{eq-condition-f_eta-f_0}, 
	we obtain 
	$\eta_1>0$, independent of $\omega_{\tau}$, such that 
	$\{\psi_\eta(t-s,\Theta_{s}\omega_{\tau}): t\geq s\}$ is a
	dynamically gradient evolution process with respect to $\{\xi_{1,\eta}^*,\cdots,\xi_{p,\eta}^*\}$, $\forall \eta \leq \eta
	_1$.
	In this case, this notion of dynamically gradient is compatible with the notion that appears in 
	\cite[Definition 4.17]{Caraballo-Langa-Liu-12}. 
\end{remark}

\begin{remark}
	 We believe that with the techniques employed in this paper it is also possible to obtain geometric structural stability, i.e., to show that Morse-Smale is stable under nonautonomous random perturbations and that there will be phase diagram isomorphism between the perturbed attractors and the limiting attractor, as we see in the deterministic case \cite[Chapter 12]{Bortolan-Carvalho-Langa-book}. This will be pursued in a future work.
\end{remark}


\section{Applications to differential equations}
\label{sec-app-dff-equations}
In this section, we present two applications. We first consider a semilinear differential equation with a small nonautonomous multiplicative white noise, and then  
 we study the effect of a small bounded noise in the damping of a damped wave equation. 
\subsection{Stochastic differential equations}\label{subsec-stochastic-diff-equations}
\par We consider the following family of stochastic 
	differential equations with a nonautonomous multiplicative 
white noise
\begin{equation}\label{eq-applications-stochastic-perturbation}
dy=Bydt+f(y)dt +\eta\kappa_ty\circ dW_t, \ t\geq \tau, \ y(\tau)=y_\tau,
\end{equation}
where $B$ is a generator of a $C^0$-semigroup 
$\{e^{Bt}: t\geq 0 \}$ on $X$, the family
$\{W_t:t\in \mathbb{R}\}$ is the standard Wiener process, see \cite{Arnold,Caraballo-Han-book}, and
$\kappa:\mathbb{R}\rightarrow \mathbb{R}$ is continuously differentiable, and
$\eta>0$. Equation \eqref{eq-applications-stochastic-perturbation} was considered
in \cite{Caraballo-Carvalho-Langa-OliveiraSousa-2021} to study hyperbolicity. 
Next, we will modify problem \eqref{eq-applications-stochastic-perturbation} to see it as a nonautonomous random differential equation satisfying the conditions of our results on the continuity and topological structure stability of attractors. 
\par The canonical sample space of a Wiener process is
$\Omega:=C_0(\mathbb{R})$ the set of continuous functions over $\mathbb{R}$ which are 
$0$ at $0$ equipped with the compact open topology. We denote $\mathcal{F}$ the associated Borel $\sigma$-algebra. Let $\mathbb{P}$ be the Wiener probability measure on 
$\mathcal{F}$ which is given by the distribution of a two-sided Wiener process with trajectories in $C_0(\mathbb{R})$. The flow $\theta$ is given by the Wiener shifts
\begin{equation*}
\theta_t\omega(\cdot)=\omega(t+\cdot)-\omega(t), \ t\in \mathbb{R}, \ \omega\in \Omega.
\end{equation*}

\begin{lemma}\label{Lemma-sublinear-growth}
	\par Consider the following scalar stochastic differential equation 
	\begin{equation}\label{eq-orstein-uhlenbeck-sde}
	dz_t+zdt=dW_t.
	\end{equation}
	There exists a $\theta$-invariant subset 
	$\tilde{\Omega}\in \mathcal{F}$ of full measure such that
	$\lim_{t\to \pm\infty}\frac{|\omega(t)|}{t}=0, \ \omega\in \tilde{\Omega}$
	and, for such $\omega$, the random variable given by
	\begin{equation*}
	z^*(\omega)=-\int_{-\infty}^0 e^{ s}\omega(s)ds
	\end{equation*}
	is well defined. Moreover,
	for $\omega\in \tilde{\Omega}$, the mapping
	$(t,\omega)\mapsto z^*(\theta_t\omega)$
	is a stationary solution of \eqref{eq-orstein-uhlenbeck-sde}
	with continuous trajectories, and 
	\begin{equation}\label{eq-sublinear-growth-OU}
	\lim_{t\to \pm\infty}\frac{|z^*(\theta_t\omega)|}{t}=0, \ \forall \,\omega\in \tilde{\Omega}.
	\end{equation}
\end{lemma}

For the proof of Lemma \ref{Lemma-sublinear-growth} see \cite[Lemma 4.1]{Caraballo-Kloeden-Schmalfu}.

\par Let $y$ be a solution for 
\eqref{eq-applications-stochastic-perturbation} and consider 
$v(t,\omega):=e^{-\eta\kappa_tz^*(\theta_t\omega)}y(t,\omega)$. Hence, $v$ has to satisfy the following nonautonomous random differential equation  
\begin{equation}\label{eq-modified-NRDE-Stochastic-ODE}
\dot{v} =Bv+e^{-\eta\kappa_tz^*(\theta_t\omega)}
f(e^{\eta\kappa_tz^*(\theta_t\omega)}v)
+\eta [\kappa_t-\dot{\kappa}_t] z^*(\theta_t\omega)v,
\end{equation}
Define $f_\eta(t,\omega,v):=e^{-\eta\kappa_tz^*(\omega)}
f(e^{\eta\kappa_tz^*(\omega)}v)+\eta [\kappa_t-\dot{\kappa}_t] z^*(\omega)v$.

\par Since the mapping $t\mapsto z^*(\theta_t\omega)$ has a sublinear growth, due to \eqref{eq-sublinear-growth-OU}, it is possible to choose a differential real function $\kappa$ for which there are random variables $m_1,m_2>0$ such that
\begin{equation*}
m_1(\omega):=\sup_{t\in \mathbb{R}}\{|\kappa_tz^*(\theta_t\omega)|\}<\infty, \hbox{ and }
m_2(\omega):=
\sup_{t\in \mathbb{R}}\{|[\kappa_t-\dot{\kappa}_t]z^*(\theta_t\omega)|\}<\infty.
\end{equation*}
\par Thus, using arguments similar to those of \cite[Section 3.3]{Caraballo-Carvalho-Langa-OliveiraSousa-2021}
we prove that the family $\{f_\eta: \eta\in [0,1]\}$ satisfies \eqref{eq-condition-f_eta-f_0}.
\par At this point, one can choose any gradient semigroup associated to 
$\dot{y}=By+f(y)$ and consider the perturbation 
$\eta\kappa_ty\circ dW_t$ and apply our results to the modified differential equation 
\eqref{eq-modified-NRDE-Stochastic-ODE}. In particular:

\begin{example}\label{example-gradient-finite-dimensional}
	Let $F:\mathbb{R}^N\to \mathbb{R}$ be a smooth real-valued function and 
	$f(x)=-\nabla F(x)$, $x\in \mathbb{R}^N$, and consider
	\begin{equation*}
	\dot{x}=f(x) +\eta \kappa_t x\circ dW_t, \ t>0.
	\end{equation*}
	When $\eta=0$ this is called a gradient system.
	Then we obtain the nonautonomous random differential equations
	\begin{equation}\label{eq-example-gradient-finite}
	\dot{x}=e^{-\eta\kappa_tz^*(\theta_t\omega)}
	f(e^{\eta\kappa_tz^*(\theta_t\omega)}x)
	+\eta [\kappa_t-\dot{\kappa}_t] z^*(\theta_t\omega)x, \ \eta\in [0,1].
	\end{equation}
	\par Assume that there exists $R_0,\sigma>0$ such that
	\begin{equation}
	f(x)\cdot x <-\sigma, \hbox{ for all }|x|\geq R_0,
	\end{equation}
	and that the set 
	$\{x\in \mathbb{R}^N: f(x)=0\}$
	is finite and consist only in hyperbolic equilibria.
	Then, $\dot{x}=f(x)$ is globally well posed and its associated with a semigroup 
	$\{T_0(t): t\geq 0\}$, which is gradient with respect to
	$\{x_1^*,\cdots,x_p^*\}$. 
	\par Then, the nonautonomous random dynamical systems associated to \eqref{eq-example-gradient-finite}
	have attractors $\{A_\eta(\omega_{\tau}): \omega_{\tau}\in \mathbb{R}\}$, and 
	this family of attractors satisfies the conclusions of Theorem 
	\ref{th-continuity-nonautonomous-random-attractors} and Theorem \ref{th-topological-structure-stability}.
\end{example}

\subsection{An application to partial differential equation}
	\label{subsec-app-dff-equations-wave}
	Now, we provide an application for a damped wave equation.
	\par Consider the damped wave equation 
	\begin{equation}\label{eq-autonomous-damped-wave-equation}
	u_{tt}+\beta u_t-\Delta u=f(u), \hbox{ in } D
	\end{equation}
	with boundary condition $u=0$, in $\partial D$, where $D$ be a bounded smooth domain in $\mathbb{R}^3$, and 
	$\beta\in (0,+\infty)$. For $f:\mathbb{R}\to \mathbb{R}$ we assume that 
	\begin{equation}\label{eq-conditions-wave-equation}
	f\in C^2(\mathbb{R}), \ \max\{|f^\prime(s)|,|f^{\prime \prime}(s)|\}\leq c(1+|s|^p), \ 
	\end{equation}
	for some $c>0$ and $p< 2$. 
	Now, we consider a small random perturbation on the damping,
		\begin{equation*}
			u_{tt}+\beta_\eta(\Theta_{t}\omega) u_t-\Delta u=f(u), \hbox{ in } D.
		\end{equation*}
	where $\beta_\eta(\omega_{\tau}):=\beta+\eta|\kappa_\tau z^*(\omega)|$, $\eta\in [0,1]$, $\omega_\tau\in \mathbb{R}\times  \Omega$, for some $\kappa$ such that $\sup_{t\in \mathbb{R}}\{|\kappa_tz^*(\theta_t\omega)|\}<\infty$, 
	Thus, there exists two $\Theta$-invariant maps 
	$b_0,b_1:\mathbb{R}\times \Omega\to (0,+\infty)$ such that
	$b_0(\omega_{\tau})\leq \beta_\eta(\Theta_t\omega_{\tau})\leq b_1(\omega_{\tau})$,
	for every $\omega_{\tau}\in \mathbb{R}\times\Omega$.
	\par The initial data will be taken in
	the space $X=H^1_0(D)\times L^2(D)$. 
	Hence, we obtain the family of abstract evolutionary equations in $X$
	\begin{equation}\label{example-eq-evolutionary-wave-eq}
	\dot{y}=B_\eta (\Theta_t \omega_{\tau}) y+ F(y), \ \eta\in [0,1]
	\end{equation}
	where
	\begin{equation*}
	y=\begin{pmatrix}
	u\\
	v 
	\end{pmatrix}\in X, \ \ 
	B_\eta(\omega_{\tau})=\begin{pmatrix}
	0 & I \\
	-A & -\beta_\eta(\omega_{\tau})
	\end{pmatrix}, \ \ 
	F(y)=\begin{pmatrix}
	0 \\
	f^e(u) 
	\end{pmatrix},
	\end{equation*}
	where $A:D(A)\subset L^2(D)\to L^2(D)$ is $-\Delta$ with Dirichlet boundary condition,
	with $f^e:H_0^1(D)\to L^2(D)$ is given by $f^e(y_1)(x)=f(y_1(x))$ for $x\in D$. 
	Thus, conditions \eqref{eq-conditions-wave-equation} implies local and global well-posedness and  
	that $f^e$ is continuously differentiable, see \cite{Arrieta-Carvalho-Hale-92} or \cite[Chapter 15]{Carvalho-Langa-Robison-book} for details.
	\par Consider the functional 
	$V:H_0^1(D)\times L^2(D)\to \mathbb{R}$ given by
	\begin{equation}
	V_0(u,v)=
	\frac{1}{2}\int_D |\nabla u|^2 \, +\frac{\beta}{2} \int_D v^2 - \int_D G(u),
	\end{equation}
	where $G(u)(x)=\int_{0}^{u(x)} f(s) \, ds$.
	Thus $V_0$ is a Lyapunov function relative to the set 
	of equilibria for \eqref{eq-autonomous-damped-wave-equation}, which we assume that is finite.
	The hyperbolic equilibrium points of \eqref{eq-autonomous-damped-wave-equation} are of the form
	$y_0^*=(u_0^*,0)$ where $u_0^*$ is a solution of 
	$-\Delta u=f(u)$
	such that $0\notin \sigma(-\Delta+D_xf^e(u_0^*)Id_X)$.
	Thus \eqref{eq-autonomous-damped-wave-equation} is associated with a gradient semigroup 
	$\{T_0(t):t\geq 0\}$, see \cite{Brunovsky-Raugel-03} for conditions to obtain that this type of dynamics is generic on damped wave equations.
%
	\par For each $y_0\in X$, $\omega_\tau\in \mathbb{R}\times \Omega$, and $\eta\in [0,1]$ Equation \eqref{example-eq-evolutionary-wave-eq} possess a unique solution which can be written as 
	\begin{equation}
	\psi_\eta(t,\omega_\tau)y_0=\varphi_\eta(t,\omega_\tau)y_0+\phi_\eta(t,\omega_\tau)y_0, \ \ t\geq 0.
	\end{equation} 
	where $\{\varphi_\eta(t,\omega): t\in [0,+\infty), \, \omega\in \Omega\}$ is the solution operator of \eqref{example-eq-evolutionary-wave-eq} with $f=0$, and 
	\begin{equation}
	\phi_\eta(t,\omega_\tau)y_0=\int_{0}^{t} \varphi_\eta (t-s,\Theta_s \omega_\tau) F(\psi_\eta(s,\omega_\tau)y_0) ds.
	\end{equation}
	\par Towards the existence of attractors, we have the following lemma. 
	\begin{lemma}\label{lemma-example-damped-wave-eq}
		 There exists a bounded subset $B$ (independent of $(t,\omega)$) which pullback attracts at time $\tau\in \mathbb{R}$, for each $\tau\leq t$, every bounded subset of $X$
			under the action of $\{\psi_\eta(t-s,\Theta_s\omega_\tau): t\geq s\}$. In particular,
			$\{\psi_\eta(t-s,\Theta_s\omega_\tau): t\geq s\}$ is strongly pullback bounded dissipative, in the sense of \cite[Definition 2.10]{Caraballo-Carv-Langa-Rivero-10}.
%
			\par Furthermore, there are $K>0$ and a $\Theta$-invariant function 
			$\alpha:\mathbb{R}\times \Omega\to (0,+\infty)$, both independent of $\eta$, such that
			\begin{equation}
			\|\varphi_\eta(t,\omega_\tau)\|_{\mathcal{L}(X)}
			\leq K e^{-\alpha(\omega_{\tau}) t}, \  t\geq 0,
			\end{equation}
			and	$\phi_\eta(t,\omega_\tau)$ is a compact operator for every $(t,\omega_\tau)\in (0,+\infty)\times\mathbb{R}\times \Omega$. In particular,  $\psi_\eta$ is pullback asymptotically compact for each $\eta\in [0,1]$, in the sense of \cite[Definition 2.14]{Bixiang-Wang-existence}.
	\end{lemma}
	The proof of Lemma \ref{lemma-example-damped-wave-eq} follows step by step the arguments presented in \cite[Section 2.1]{Caraballo-Carv-Langa-Rivero-10} (or see \cite[Chapter 15]{Carvalho-Langa-Robison-book} for more detailed proofs), thus it will be omitted.
	Thus there are nonautonomous random attractors
	$\{\mathcal{A}_\eta(\omega_{\tau}): \omega_{\tau}\in \mathbb{R}\times \Omega\}$ for 
	$(\psi_{\eta},\Theta)$ for all $\eta\in [0,1]$ satisfying Condition 
	\textbf{(a.1)} of Remark \ref{remark-colectiv-asympt-compact}, see \cite{Bixiang-Wang-existence,Caraballo-Carv-Langa-Rivero-10}. 
	Additionally, using arguments similar to those in \cite{Caraballo-Carv-Langa-Rivero-10}
	the family $\{(\psi_\eta,\Theta)\}_{\eta\in [0,1]}$
	is collectively pullback asymptotically compact at $\eta=0$. Therefore, conditions of Remark \ref{remark-colectiv-asympt-compact} are satisfied and
	it is possible to apply our results to conclude that the family of attractors behaves continuously (using Theorem \ref{th-continuity-nonautonomous-random-attractors}) and that 
	we have topological structural stability (using Theorem \ref{th-topological-structure-stability}). 

	\begin{remark}\label{remark-example-bounded}
		Instead of considering $\beta_\eta(\omega_{\tau}):=\beta+\eta|\kappa_\tau z^*(\omega)|$, we could have considered the following perturbation
		\begin{equation}
		\tilde{\beta}_\eta(\omega)=\beta+\eta \frac{2}{\pi}\arctan\circ z^*(\omega),
		\ \omega\in \Omega, 
		\end{equation}
		for $\beta\in (1,+\infty)$.
		For this perturbations a condition as \eqref{eq-condition-f_eta-f_0-uniformly-on-Omega_R} is verify for the symbol space $\Omega$ instead of $\mathbb{R}\times \Omega$.
		See also \cite{Caraballo-Cruz-21} where the authors study this type of perturbations. 
	\end{remark}

\section*{Acknowledgments}
We acknowledge the financial support from the following institutions: T. Caraballo and J. A. Langa by Ministerio de Ciencia, Innovaci\'on y Universidades (Spain), FEDER (European Community) under grant PGC2018-096540-B-I00, and by FEDER Andaluc\'{\i}a Proyecto I+D+i Programa Operativo US-1254251, and Proyecto PAIDI 2020 P20 00592; A. N. Carvalho by S\~ao Paulo Research Foundation (FAPESP) grant 2020/14075-6, CNPq grant 306213/2019-2, and FEDER - Andalucía P18-FR-4509; and A. N. Oliveira-Sousa by S\~ao Paulo Research Foundation (FAPESP) grants 2017/21729-0 and 2018/10633-4, and CAPES grant PROEX-9430931/D.

\bibliographystyle{abbrv}
\bibliography{references_CNRA}

\end{document}